\documentclass[11pt]{article}

\usepackage{arxiv}

\usepackage[utf8]{inputenc} 
\usepackage[T1]{fontenc}    
\usepackage{hyperref}       
\usepackage{url}            
\usepackage{booktabs}       
\usepackage{amsfonts}       
\usepackage{nicefrac}       
\usepackage{microtype}      
\usepackage{amsthm}

\usepackage{epstopdf}
\usepackage{algorithmic}
\usepackage{amsmath, amssymb, color} 
\usepackage{cleveref}       
\usepackage{lipsum}         
\usepackage{graphicx}
\usepackage{doi}
\usepackage[shortlabels]{enumitem}
\usepackage{makecell}
\usepackage{multicol}
\usepackage{rotating}

\usepackage{array,amscd}
\usepackage{amsfonts}
\usepackage[T1]{fontenc} 
\usepackage{enumerate}
\usepackage[arrow, matrix, curve]{xy}
\usepackage[round,sort,authoryear]{natbib}
\usepackage{siunitx}
\usepackage{multirow}
\usepackage[font=small,labelfont=bf]{caption}
\usepackage{blkarray}
\newcommand*{\email}[1]{\href{mailto:#1}{\nolinkurl{#1}} } 

\usepackage{tikz}
\usetikzlibrary{decorations}
\usetikzlibrary {shapes}
\usetikzlibrary {arrows}
\usetikzlibrary {positioning}
\usetikzlibrary {calc}

\ifpdf
  \DeclareGraphicsExtensions{.eps,.pdf,.png,.jpg}
\else
  \DeclareGraphicsExtensions{.eps}
\fi

\definecolor{darkgreen}{rgb}{0,0.4,0}
\definecolor{MyBlue}{rgb}{0,0.08,0.7} 
\definecolor{MyRed}{rgb}{0.85,0.08,0}

\makeatletter
\renewcommand*\env@matrix[1][\arraystretch]{%
  \edef\arraystretch{#1}%
  \hskip -\arraycolsep
  \let\@ifnextchar\new@ifnextchar
  \array{*\c@MaxMatrixCols c}}
\makeatother

\newcommand{\PD}{\mathrm{PD}}

\newtheorem{theorem}{Theorem}
\newtheorem{proposition}[theorem]{Proposition}
\newtheorem{lemma}[theorem]{Lemma}
\newtheorem{corollary}[theorem]{Corollary}
\newtheorem{definition}[theorem]{Definition}

\newtheorem{remark}[theorem]{Remark}
\newtheorem{example}[theorem]{Example}

\title{Identifiability in Continuous Lyapunov Models}

\author{{Philipp Dettling}\thanks{Technical University of Munich; TUM School of Computation, Information and Technology, Department of Mathematics and Munich Data Science Institute  (\email{philipp.dettling@tum.de}, \email{roser.homs@tum.de},  \email{mathias.drton@tum.de}).}
\And {Roser Homs}\footnotemark[1]
\And {Carlos Améndola}\thanks{Technical University of Berlin, Institute of Mathematics; Max Planck Institute for Mathematics in the Sciences, Leipzig, Germany (\email{amendola@math.tu-berlin.de})}
\And {Mathias Drton}\footnotemark[1]
\And
{Niels Richard Hansen}\thanks{University of Copenhagen, Denmark; Department of Mathematical Sciences
(\email{niels.r.hansen@math.ku.dk}).}}

\renewcommand{\headeright}{}
\renewcommand{\undertitle}{}
\renewcommand{\shorttitle}{Identifiability in Continuous Lyapunov Models}

\begin{document}

\maketitle

\begin{abstract}
    The recently introduced graphical continuous Lyapunov models provide a new approach to statistical modeling of correlated multivariate data.  The models view each observation as a one-time cross-sectional snapshot of a multivariate dynamic process in equilibrium.  The covariance matrix for the data is obtained by solving a continuous Lyapunov equation that is parametrized by the drift matrix of the dynamic process.  In this context, different statistical models postulate different sparsity patterns in the drift matrix, and it becomes a crucial problem to clarify whether a given sparsity assumption allows one to uniquely recover the drift matrix parameters from the covariance matrix of the data.  We study this identifiability problem by representing sparsity patterns by directed graphs.  Our main result  proves that the drift matrix is globally identifiable if and only if the graph for the sparsity pattern is simple (i.e., does not contain directed two-cycles).  Moreover, we present a necessary condition for generic identifiability and provide a computational classification of small graphs with up to 5 nodes.
\end{abstract}

\section{Introduction}

Recent work of \cite{katie2019} and \cite{hansen2020} introduces Lyapunov models as a new paradigm of probabilistic graphical modeling \citep{handbook:graphical:models}.
When capturing cause-effect relations among observations, standard graphical models directly postulate noisy functional relations among the considered random variables \citep{pearl:2009,spirtes:2000,peters2017}.  In contrast, the new Lyapunov models introduce a temporal perspective that simplifies, in particular, modeling of feedback loops.  Suppose the data at hand are collected by observing a $p$-dimensional random vector.  
Lyapunov models assume the random vector to arise as a one-time cross-sectional observation of a $p$-dimensional dynamic process in equilibrium.    
When working in continuous time, the natural model for the process is an Ornstein-Uhlenbeck process $X(t)$ that is given by the stochastic differential equation
\begin{equation}
\label{eq:pdimOU}
    \mathrm{d}X(t)=M(X(t)-a)\,\mathrm{d}t+ D\, \mathrm{d}W(t),
\end{equation}
where $W(t)$ is a Wiener process, and $a \in \mathbb{R}^{p}$ and $M,D \in \mathbb{R}^{p \times p}$ are non-singular parameter matrices. In this context, the matrix $M$ is a drift matrix that quantifies temporal cause-effect relations among the variables, and $C=DD^\top$ is a positive definite volatility matrix.  

If $M$ is a stable matrix, i.e., all its eigenvalues have negative real part, then $X(t)$ has a stationary Gaussian distribution $N(a,\Sigma)$, where the covariance matrix $\Sigma$ is the unique solution \citep{risken1996} to the continuous Lyapunov equation
\begin{equation}
\label{eq:lyapunoveq}
    M\Sigma+\Sigma M^{\top}+C=0.
\end{equation}

A graphical continuous Lyapunov model as defined by \cite{katie2019} and \cite{hansen2020} refines this setup by assuming that the drift matrix $M=(m_{ij})$ exhibits a specific zero pattern that is given by a directed graph.  
A similar perspective was presented by \cite{Young:2019} for discrete time 
autoregressive
models, which leads to an equilibrium
covariance matrix solving the \emph{discrete} Lyapunov equation.
Based on an estimated covariance matrix $\hat{\Sigma}$, both \cite{katie2019} and \cite{hansen2020} develop estimation techniques for $M$ using \eqref{eq:pdimOU}. They consider a setting where the data are comprised of a sample of independent and identically distributed random observation vectors that are obtained from several independent copies of the multivariate Ornstein-Uhlenbeck process in equilibrium. An application is shown in both works using the data set of \cite{Sachs2005}. These data contains measurements of expression of different proteins in human immune system cells that are harvested and subjected to flow cytometry (and thus `destroyed'). The estimation methods of  \cite{katie2019} and \cite{hansen2020} are applied to obtain estimates of the protein signaling network, recovering substantial parts of the version accepted among biologists. Naturally, statisticians then aim to provide theoretical guarantees for the methods applied.  Under the assumption of parameter identifiability, a consistency result for the estimation method of \cite{katie2019} is derived in \cite{Dettling2022}. Parameter identifiability means that given a fixed support of the matrix $M$ and positive definite matrices $C,\Sigma$, we are able to uniquely recover the entries in $M$. This central question for statistical theory for Lyapunov models is the object of study in this work.

\subsection*{Organization and results of the paper}

In \Cref{sec:preliminaries} we introduce graphical continuous Lyapunov models and motivate the question of identifiability with the help of the directed 3-cycle as a running example.
In \Cref{sec:identifiabilityASigma} we formally introduce the notions of generic and global identifiability and make some preliminary observations. In \Cref{sec:asigmastru}, we explain the structure of the matrix $A(\Sigma)$ that arises from (half-)vectorization of the Lyapunov equation.  We also highlight how the rank of a submatrix of $A(\Sigma)$ determines generic and global identifiability of a model.  Exploiting block structure in the relevant submatrix of $A(\Sigma)$, we prove global identifiability for all directed acyclic graphs (DAGs) in \Cref{sec:DAGs}. Our proof also yields that the models given by DAGs are closed algebraic subsets of $\PD_p$, and that the models associated to complete DAGs are equal to $\PD_p$
 (\Cref{cor:completeModel}).  In \Cref{sec:cyclic}, we turn to cyclic graphs for which the relevant matrices no longer exhibit block structure.  We demonstrate that for small graphs the approach studying factorizations of determinants can still be implemented using sum of squares methods to certify that the relevant polynomials are positive on $\PD_p$.  
 In \Cref{sec:stable} we present our main result (\Cref{th:simple}), which proves that global model identifiability holds if the underlying graph is simple (i.e., does not contain any 2-cycle).  If $C$ is diagonal---the case of primary practical interest---then the requirement that the graph be simple is also necessary for global identifiability.  Moreover, we are able to show that for all $C\in\PD_p$, all simple graphs yield models $\mathcal{M}_{G,C}$ that are closed algebraic subsets  of $\PD_p$. We discuss further the diagonal hypothesis on $C$ in \Cref{sec:C}. 
In \Cref{sec:nonsimple}, we turn to the weaker notion of generic identifiability, for which we develop a necessary criterion and computationally classify all non-simple graphs with up to 5 nodes.
The paper concludes in \Cref{sec:conclusion}. Some details on the structure of the matrix $A(\Sigma)$ and the factorization of its minors are deferred to \Cref{sec:app}.

The code we used for our computations is available at the repository website \url{https://mathrepo.mis.mpg.de/LyapunovIdentifiability}.

\section{Preliminaries}
\label{sec:preliminaries}

A graphical continuous Lyapunov model as defined in \cite{Dettling2022} considers the setup that the drift matrix $M=(m_{ij})$ exhibits a specific zero pattern that is given by a directed graph $G$ on the set of nodes $[p]=\{1,\dots,p\}$, with $m_{ji}=0$ whenever $i\to j$ is not an edge in $G$. In this setting our graphs will always include self-loops $i\to i$. 

\begin{example}
\label{exa:graphM}
The directed 3-cycle $G$ with vertex set $V=\lbrace 1,2,3 \rbrace$ and edge set $E= \lbrace 1\to 1, 2 \to 2, 3 \to 3, 1\to 2, 2\to 3, 3 \to 1 \rbrace$, which is displayed in \Cref{fig:cycle}, encodes drift matrices of the form
$$M = \begin{pmatrix}
m_{11} & 0 & m_{13} \\
m_{21} & m_{22} & 0 \\
0 & m_{32} & m_{33} 
\end{pmatrix}.$$
\end{example}

\begin{figure}[t]
\begin{center}
\begin{tikzpicture}[->,every node/.style={circle,draw},line width=1pt]
  \node (1) at (1.5,1.0) {1};
  \node (2) at (0,0) {$2$};
  \node (3) at (3,0) {$3$};
\foreach \from/\to in {1/2,2/3,3/1}
\draw (\from) -- (\to);   
\path (1) edge [loop above] (1);
\path (2) edge [loop left] (2);
\path (3) edge [loop right] (3);
 \end{tikzpicture}  
\caption{The directed 3-cycle.}
\label{fig:cycle}
\end{center}
\end{figure}
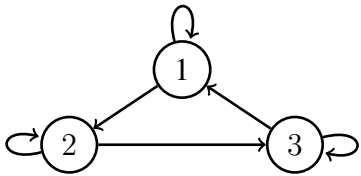

The Lyapunov equation from \eqref{eq:lyapunoveq} is a symmetric matrix equation providing $p(p+1)/2$ constraints.  In contrast, the drift matrix $M$ is a $p\times p$ matrix that need not be symmetric.  Hence, without any assumptions on its structure, $M$ is never uniquely determined by the covariance matrix $\Sigma$ of the observations.
For graphical Lyapunov models, this leads to a key identifiability question:  For which sparsity patterns 
can the drift matrix $M$ be recovered from the positive definite covariance matrix $\Sigma$?   Our treatment of this question will assume that the volatility matrix $C$ is a known positive definite matrix. 
While some of our results hold for all positive definite $C$, others require the assumption that $C$ is diagonal. This is a sensible assumption as it corresponds to the setting of uncorrelated noise. A special case is the assumption $C=2I_p$ that covers the natural setting of homoscedastic noise.

\begin{remark}
Evidently, if a matrix $\Sigma$ solves the Lyapunov equation \eqref{eq:lyapunoveq} for a pair $(M,C)$ then $\Sigma$ also solves the equation given by $(\gamma M,\gamma C)$ for any $\gamma\in\mathbb{R}$. An implication of this fact is that our results on recovery of $M$ for fixed $C$ also address the setting of models in which $C=\gamma C'$, with $C'$ known and positive definite but $\gamma>0$ an unknown parameter.  In this latter setting, one can only hope to recover $M$ up to a scalar multiple and this is possible if and only if $M$ can be recovered uniquely in the setting where we fix $C=C'$.
\end{remark}

Before proceeding to illustrate the identifiability problem for \Cref{exa:graphM}, we give a formal definition of graphical continuous Lyapunov models as sets of covariance matrices.  We 
write $\PD_p$ for the cone of $p \times p$ positive definite matrices.

 \begin{definition}\label{def:model}
Let $G=(V,E)$ be a directed graph with vertex set $V=[p]$ and an edge set $E$ that includes all self-loops $i\to i$, $i\in[p]$. We write
$\mathbb{R}^E$ for the space of matrices $M=(m_{ij})\in\mathbb{R}^{p \times p}$ with $m_{ji}=0$ whenever $i\to j \notin E$.    Given a choice of $C\in\PD_p$, the \emph{graphical continuous Lyapunov model} of $G$ is the set of covariance matrices
\begin{equation*}\label{eqn:modelG}
    \mathcal{M}_{G,C} = \big\lbrace \Sigma \in \PD_p \,  : \, M\Sigma +\Sigma M^{\top}= -C\, \text{ for some }  M \in \mathbb{R}^E  \big\rbrace ,  
\end{equation*} 
\end{definition}

\begin{remark}\label{rem:dropStability}
Let $\mathrm{Stab}(E)\subseteq\mathbb{R}^E$ be the subset of stable matrices, which is always non-empty and open, since $E$ is assumed to always include all self-loops, see Definition \ref{def:model}. In particular, this implies that $\dim (\mathrm{Stab}(E))=|E|$.
When $C$ is positive definite, the Lyapunov equation from \eqref{eq:lyapunoveq} has a positive definite solution $\Sigma$ if and only if $M$ is stable \citep[Theorem 1.1]{bhaya2003}.  Hence, the definition of the model $\mathcal{M}_{G,C}$ remains unchanged if we replace the requirement $M\in\mathbb{R}^E$ by $M\in\mathrm{Stab}(E)$.
\end{remark}

The identifiability question we pose asks if a covariance matrix $\Sigma$ in the model $\mathcal{M}_{G,C}$ may simultaneously solve the Lyapunov equation for more than one choice of a matrix $M\in\mathbb{R}^E$.  In other words, we study the injectivity of the (rational) parametrization map
\begin{equation}\label{eq:map}
\begin{split}
    \phi_{G,C}:\, \mathrm{Stab}(E) &\rightarrow \PD_p\\
        M&\mapsto \Sigma(M,C),
\end{split}        
\end{equation}
where $\Sigma(M,C)$ is the unique matrix $\Sigma$ that solves the Lyapunov equation given by the stable matrix $M$ and positive definite $C$. See \eqref{eq:lyapunoveq usual vec} for details on this uniqueness. 

By vectorization, the Lyapunov equation~\eqref{eq:lyapunoveq} is transformed into the linear equation system  
\begin{equation}
\label{eq:ASigma:p3}
A(\Sigma) \mathrm{vec}(M) = -\mathrm{vech}(C),
\end{equation}
where $\mathrm{vech}(C)$ is the half-vectorization of a fixed symmetric matrix $C\in\PD_p$, 
and $A(\Sigma)$ is a $p(p+1)/2 \times p^2$ matrix depending on $\Sigma$ whose form will be discussed in \Cref{sec:asigmastru}.

\begin{example}
\label{ex:ASigma3}
In the case of $p=3$ variables the matrix $A(\Sigma)$ equals
\begin{align*}
\begin{blockarray}{@{\hspace{10pt}}c@{\hspace{10pt}}c@{\hspace{10pt}}c@{\hspace{10pt}}c@{\hspace{10pt}}c@{\hspace{10pt}}c@{\hspace{10pt}}c@{\hspace{10pt}}c@{\hspace{10pt}}c@{\hspace{10pt}}c}
& 1\to 1 & 1\to 2 & 1\to 3 & 2 \to 1 & 2 \to 2 & 2 \to 3 & 3 \to 1 & 3 \to 2 & 3 \to 3 \\
\begin{block}{c(ccccccccc)}
            (1,1)&2 \Sigma_{11} &0&0&2 \Sigma_{12} & 0 &0 & 2\Sigma_{13} & 0 & 0 \\
            (1,2)&\Sigma_{12} & \Sigma_{11} &0& \Sigma_{22}  & \Sigma_{12} &0& \Sigma_{23} & \Sigma_{13} &0\\
            (1,3)&\Sigma_{13} &0&\Sigma_{11}& \Sigma_{23} & 0 & \Sigma_{12}& \Sigma_{33} & 0 & \Sigma_{13}\\
            (2,2)&0 &2\Sigma_{12}& 0& 0 & 2 \Sigma_{22} &0& 0 & 2\Sigma_{23} & 0 \\
            (2,3)&0 & \Sigma_{13}&\Sigma_{12}& 0 & \Sigma_{23}& \Sigma_{22}   & 0 & \Sigma_{33} & \Sigma_{23}\\
            (3,3)&0 & 0&2\Sigma_{13}& 0 & 0&2\Sigma_{23} & 0 & 0 & 2 \Sigma_{33}\\
\end{block}
\end{blockarray},
\end{align*}
where the column index $i\to j$ corresponds to entry $m_{ji}$ of the drift matrix $M=(m_{ij})$.
\end{example}

Given a graph $G$ with $p(p+1)/2$ edges, unique solvability of~\eqref{eq:ASigma:p3} for $M\in\mathbb{R}^E$ is equivalent to a certain maximal square submatrix of $A(\Sigma)$ being invertible. This submatrix is formed by all columns of $A(\Sigma)$ corresponding to edges of the graph. Observe that two columns indexed by $i\to j$ and $k\to l$ have the same zero pattern whenever $j=l$. This motivates ordering the columns of $A(\Sigma)_{\cdot,E}$ increasingly with
\begin{align}\label{eq:col_order}
    i \to j < k \to l  \quad \text{ if } j<l \text{ or } j=l,i<k.  
\end{align}
\noindent 
Moreover, note that for simple graphs there is a natural pairing between pairs $(i,j)$ with $i\leq j$ and edges between $i$ and $j$. In this case, we will order rows accordingly with their corresponding pair $(i,j)$.

\begin{example}\label{ex:Identifiability3cycle}
Consider the 3-cycle $G$ from \Cref{exa:graphM}. The submatrix of $A(\Sigma)$ in \Cref{ex:ASigma3} associated to $G$ is
\begin{equation*}
    A(\Sigma)_{\cdot,E}=
\begin{blockarray}{ccccccc}
&1 \to 1 & 3 \to 1 & 1 \to 2 & 2 \to 2 & 2 \to 3   & 3 \to 3 \\
\begin{block}{c(cccccc)}
   (1,1) & 2\Sigma_{11}&2\Sigma_{13}&0&0&0&0\\
   (1,3) & \Sigma_{13}&\Sigma_{33}&0&0&\Sigma_{12}&\Sigma_{13}\\
   (1,2) &  \Sigma_{12}&\Sigma_{23}&\Sigma_{11}&\Sigma_{12}&0&0\\
   (2,2) & 0&0&2\Sigma_{12}&2\Sigma_{22}&0&0\\
   (2,3) & 0&0&\Sigma_{13}&\Sigma_{23}&\Sigma_{22}&\Sigma_{23}\\
   (3,3) & 0&0&0&0&2\Sigma_{23}&2\Sigma_{33}\\
    \end{block}
\end{blockarray}.
\end{equation*}
To show invertibility of $A(\Sigma)_{\cdot,E}$, we may inspect its determinant, which factorizes as 
\begin{equation}
\label{eq:detASigma3cycle}
 |\det(A(\Sigma)_{\cdot,E})|= 2^3 \cdot \det(\Sigma) \cdot (\Sigma_{11}\Sigma_{22}\Sigma_{33}-\Sigma_{12}\Sigma_{13}\Sigma_{23}).
\end{equation}
All displayed factors are positive when $\Sigma$ is positive definite. Indeed,  $\det(\Sigma)>0$ and the fact that $\det(\Sigma_{ij,ij})=\Sigma_{ii}\Sigma_{jj}-\Sigma_{ij}^2>0$ for all $i\neq j$ implies that $\Sigma_{11}^2\Sigma_{22}^2\Sigma_{33}^2 > \Sigma_{12}^2\Sigma_{13}^2\Sigma_{23}^2$, which clarifies that the last factor is also positive.
Alternatively, we can show this using the identity
\begin{align*}
(\Sigma_{11}&\Sigma_{22}\Sigma_{33})^2 -(\Sigma_{12}\Sigma_{13}\Sigma_{23})^2 = \\
&(\Sigma_{13}\Sigma_{23})^2 \det(\Sigma_{12,12}) + \Sigma_{11}\Sigma_{22}\Sigma_{23}^2 \det(\Sigma_{13,13}) + \Sigma_{11}^2\Sigma_{22}\Sigma_{33} \det(\Sigma_{23,23}) > 0. \end{align*}
We conclude that when $G$ is the 3-cycle, then for \emph{all} covariance matrices $\Sigma\in\mathcal{M}_{G,C} \subseteq \mathrm{PD}_3$ there is a unique matrix $M\in\mathbb{R}^E$ such that $\Sigma=\phi_{G,C}(M)$.  We will refer to this property as the 3-cycle defining a \emph{globally identifiable} model.  Note that our argument also shows that $\mathcal{M}_{G,C} = \mathrm{PD}_3$.
\end{example}

This small example already reveals some of the subtleties arising when analyzing identifiability of continuous Lyapunov models.  The problem can be reduced to determining whether a particular submatrix that is sparsely populated with covariances has full rank (see \Cref{lem:rankidentifiability} and \Cref{theo:kernelcrit}) but the resulting matrices have involved graph-dependent structures. The choice of ordering in \eqref{eq:col_order} is especially insightful for directed acyclic graphs. After sorting the nodes such that if $i\to j$ then $i\leq j$, any DAG yields a block upper-triangular matrix, as in \Cref{ex:complete DAG 3 nodes}, from which identifiability for all associated models follows (\Cref{theo:dagidenti}).
For cyclic graphs, however, the polynomials that appear while factoring determinants, as in \eqref{eq:detASigma3cycle}, quickly increase in complexity, and it is not easy to determine whether they are non-zero.  In our main result (\Cref{th:simple}) we thus consider alternative spectral arguments that use the stability of the drift matrix $M$ in order to derive identifiability.

\section{Notions of identifiability}
\label{sec:identifiabilityASigma}

We begin by recalling the concept of fibers that is useful to define the different notions of identifiability we study in subsequent sections.  Let $C\in\PD_p$, and let $\mathcal{M}_{G,C}$ be the graphical continuous Lyapunov model associated to a directed graph $G=(V,E)$ with vertex set $V=[p]$ and edge set $E$.  Let $\phi_{G,C}$ be the parametrization from~\eqref{eq:map}.  The \emph{fiber} of a matrix $M_0 \in \mathrm{Stab}(E)$ is the set 
\begin{equation}
\label{eq:fiber}
    \mathcal{F}_{G,C}(M_0)=\lbrace M\in \mathrm{Stab}(E): \phi_{G,C}(M)=\phi_{G,C}(M_0)\rbrace.
\end{equation}
In other words, a fiber comprises all drift matrices $M\in\mathbb{R}^E$ whose Lyapunov equation (for the fixed matrix $C\in\PD_p$) is solved by a given covariance matrix $\Sigma$.

We will consider three natural notions of identifiability.

\begin{definition}
\label{def:notions_identifiability}
Let $\mathcal{M}_{G,C}$ be the graphical continuous Lyapunov model given by a directed graph $G=(V,E)$ with $V=[p]$ and $C\in\PD_p$. The model $\mathcal{M}_{G,C}$ is 
\begin{itemize}
\item[(i)] \emph{globally identifiable} if $\mathcal{F}_{G,C}(M_{0})=\{M_0\}$ for all $M_0 \in \mathrm{Stab}(E)$;  
\item[(ii)] \emph{generically identifiable} if $\mathcal{F}_{G,C}(M_{0})=\{M_0\}$ for almost all $M_0 \in \mathrm{Stab}(E)$, i.e., the matrices with $\mathcal{F}_{G,C}(M_{0})\not=\{M_0\}$ form a 
Lebesgue null set in 
$\mathbb{R}^E$;
\item[(iii)] \emph{non-identifiable} if $|\mathcal{F}_{G,C}(M_{0})|=\infty$ for all $M_0 \in \mathrm{Stab}(E)$.
\end{itemize}
\end{definition}

\begin{remark}
The generic properties we prove in this paper are 
derived by showing that they hold outside  
a strict subset of $\mathrm{Stab}(E)$ that is described by polynomials in the entries of the drift matrix; see e.g.~\Cref{lem:rankidentifiability}.  Hence, in a generically identifiable model the exception set is not merely a set of Lebesgue measure zero, but also a lower-dimensional algebraic subset of $\mathrm{Stab}(E)$.
\end{remark}

\begin{remark}
Characterizing identifiability is also a key problem for standard directed graphical models; see  \cite{drton2018} and \citet[Chap.~16]{Sullivant2018} for a discussion of the different notions of identifiability in this context.  
For standard graphical models, necessary and sufficient conditions for global identifiability have been obtained \citep{Drton2011}.  However, many models of interest are not globally identifiable, and much 
 work has also gone into criteria for generic identifiability 
\citep{brito2006graphical, kumor2019, foygel2012half,drton2016generic}. \end{remark}

The 3-cycle from  \Cref{ex:Identifiability3cycle} is an example of global identifiability.  Under global identifiability, no two distinct stable matrices may define the same covariance matrix in the model given by the graph. Unfortunately, this is not always the case.  

\begin{example}
\label{ex:2cycle}
Consider the 2-cycle $G=(V,E)$ with $V=\lbrace 1,2 \rbrace$ and $E=\lbrace 1\to 1, 2\to 2, 1 \to 2, 2\to 1 \rbrace $.  Then $\phi_{G,C}$ maps the 4-dimensional parameter space $\mathrm{Stab}(E)$ to the 3-dimensional $\mathrm{PD}_2$-cone.  Hence, when computing any fiber we have to solve a linear system that is underdetermined, with 3 equations in 4 unknowns.  Therefore, 
$\mathcal{M}_{G,C}$ is non-identifiable, no matter the choice of $C\in\PD_2$.
\end{example}

The example just given generalizes as follows:
\begin{lemma}\label{lem:dim}
Let $G=(V,E)$ be a directed graph with vertex set $V=[p]$, and let $C\in\PD_p$. If $|E|>\dim(\mathcal{M}_{G,C})$, i.e., the number of free parameters in $\mathrm{Stab}(E)$ is greater than the dimension of the model, then $\mathcal{M}_{G,C}$ is non-identifiable.
In particular, all graphs with $|E|> p(p+1)/2$ give non-identifiable models.

\begin{proof}
By the Hurwitz criterion, the set of sparse stable matrices $\mathrm{Stab}(E)$ is semialgebraic, see \citet[Theorem 2.3.3]{horn_johnson_1991}.  As its dimension is 
 $\dim(\mathrm{Stab}(E))=\vert E\vert>\dim(\mathcal{M}_{G,C})$, it follows that
the rational map $\phi_{G,C}$ defined on $\mathrm{Stab}(E)$ is generically infinite-to-one; see, e.g., \citet[Lemma 2.5]{BDSW22}.  Apply \cref{lem:rankidentifiability} below to conclude that all fibers are infinite.
\end{proof}
\end{lemma}

A straightforward but very useful fact when studying global identifiability is that if a graph $G=(V,E)$ yields a globally identifiable model then so does every subgraph $H=(V,E')$, $E' \subseteq E$, that is obtained by removing edges of the form $i\rightarrow j$ with $i\neq j$. We record this fact as:
\begin{proposition}
\label{prop:subgraph}
Let $\mathcal{M}_{G,C}$ be a globally identifiable model given by a directed graph $G=(V,E)$ with $V=[p]$ and $C\in\PD_p$. Let $E'\subset E$ be a subset of the edges.  Then the model $\mathcal{M}_{H,C}$ defined by the subgraph $H=(V,E')$ is globally identifiable.
\end{proposition} 
\begin{proof}
It holds that $\mathrm{Stab}(E')\subseteq\mathrm{Stab}(E)$.  Therefore, for every matrix $M_0\in\mathrm{Stab}(E')$, we have $\mathcal{F}_{H,C}(M_0)\subseteq \mathcal{F}_{G,C}(M_0)=\{M_0\}$, where the last equality is due to the assumed global identifiability of $\mathcal{M}_{G,C}$.
\end{proof}

In the case where $C$ is diagonal, further conclusions can be made.

\begin{proposition}
\label{prop:Cdiag}
Let $G=(V,E)$ be a directed graph with $V=[p]$. Let $C\in\PD_p$ be diagonal, and let $I_p$ be the $p\times p$ identity matrix.  Then the models for $C$ versus $I_p$ are isomorphic, and so are their fibers:
\begin{itemize}
    \item[(i)] $\mathcal{M}_{G,C}=C^{1/2}\mathcal{M}_{G,I_p} C^{1/2}$, and
    \item[(ii)] $\mathcal{F}_{G,C}(M)=\mathcal{F}_{G,I_p}(C^{1/2}MC^{-1/2})$ for all $M\in\mathrm{Stab}(E)$.
\end{itemize}
In particular, $\mathcal{M}_{G,C}$ is globally/generically identifiable if and only if  $\mathcal{M}_{G,I_p}$ is globally/generically identifiable. \end{proposition}
\begin{proof}
Since $C$ is diagonal, the similarity transformation $\tau_1:M\mapsto C^{-1/2}MC^{1/2}$ is an automorphism of $\mathbb{R}^E$, with  $\tau_1(\mathrm{Stab}(E))=\mathrm{Stab}(E)$.  Define a second linear map $\tau_2:\Sigma\mapsto C^{-1/2}\Sigma C^{-1/2}$, an automorphism of the space of symmetric matrices with $\tau_2(\PD_p)=\PD_p$.  Now
\begin{multline*}
 M\Sigma + \Sigma M^\top +C=0 \iff \\
(C^{-1/2}MC^{1/2})(C^{-1/2}\Sigma C^{-1/2}) + (C^{-1/2}\Sigma C^{-1/2})( C^{-1/2} M C^{1/2})^\top+I_p=0.
\end{multline*}
Thus, $\mathcal{M}_{G,I_p}=\tau_2(\mathcal{M}_{G,C})$ and 
$\mathcal{F}_{G,I_p}(M)=\mathcal{F}_{G,C}(\tau_1^{-1}(M))$.
\end{proof}

In \Cref{prop:subgraph} only edges are removed when forming a subgraph. When $C$ is diagonal we may strengthen the result to subgraphs in which we also remove vertices; compare  \citet[Lemma 1]{Drton2011} in the context of standard graphical models.

\begin{proposition}
\label{prop:subgraph:Cdiag}
Let $G=(V,E)$ be a directed graph with $V=[p]$, and let $H=(V',E')$ be a subgraph with $V' \subseteq V$ and $E' \subseteq E$.  If the model $\mathcal{M}_{G,C}$ is globally identifiable for a diagonal matrix $C\in\PD_p$, then $\mathcal{M}_{H,C'}$ is globally identifiable for all diagonal matrices $C'\in\PD_{p'}$, where $p'=|V'|$.
\end{proposition} 
\begin{proof}
By \Cref{prop:subgraph}, it suffices to prove that removing an isolated vertex from $G$ preserves global identifiability of the model for $C$ diagonal. 
By \Cref{prop:Cdiag}, we may assume that $C=I_p$ and $C'=I_{p-1}$, where $p$ is an isolated node of $G$. 
Let $M\in\mathrm{Stab}(E)$, and let $M_{[p-1],[p-1]}$ be the submatrix comprising the first $p-1$ rows and columns.  
Since $p$ is isolated,
the $p$th row and column of $M$ is zero with the exception of the diagonal entry $m_{pp}$.  It is not difficult to see that  $\Sigma=\phi_{G,I_p}(M)$ also has its $p$th row and column equal to zero except for the diagonal entry which equals $\Sigma_{pp}=-1/(2m_{pp})$.  Hence, the entry $m_{pp}$ is always uniquely determined by $\Sigma$, and we conclude that the cardinality of the fiber $\mathcal{F}_{G,I_p}(M)$ is equal to the cardinality of $\mathcal{F}_{H,I_{p-1}}(M_{[p-1],[p-1]})$.  Since every matrix in $\mathrm{Stab}(E')$ is a submatrix $M_{[p-1],[p-1]}$ of a matrix $M\in\mathrm{Stab}(E)$, the model $\mathcal{M}_{H,I_{p-1}}$ is globally identifiable.  
\end{proof}

Combining \Cref{prop:subgraph:Cdiag} with \Cref{ex:2cycle}, we obtain that the graph of a globally identifiable model cannot contain any 2-cycles.  

\begin{definition}
A directed graph $G=(V,E)$ is \emph{simple} if it is free of 2-cycles, i.e., there do not exist two distinct nodes $i,j \in V$ such that $i \to j \in E$ and $j \to i \in E$. Otherwise, we call $G$ non-simple. 
\end{definition}

\begin{proposition}
\label{prop:non-simple}
If a directed graph $G=(V,E)$, $V=[p]$, defines a globally identifiable model $\mathcal{M}_{G,C}$ when $C\in\PD_p$ is diagonal, then $G$ must be simple.
\end{proposition} 

\begin{remark}
 \Cref{prop:subgraph:Cdiag} and \Cref{prop:non-simple} may fail for non-diagonal $C\in\PD_p$. See \Cref{sec:C} for an example.
\end{remark}

Unfortunately, similar subgraph arguments cannot be made for generic instead of global identifiability. Indeed, generic identifiability may be lost but also restored when removing an edge.  \Cref{ex:subgraphnotiden} illustrates this phenomenon.

\section{Rank conditions}
\label{sec:asigmastru}

In this section, we discuss solving the Lyapunov equation \eqref{eq:lyapunoveq} for the generally non-symmetric drift matrix $M$ given the symmetric matrices $\Sigma$ and $C$.  We will proceed by vectorizing the Lyapunov equation, and we will
state necessary and sufficient conditions for identifiability based on the ranks of submatrices of the coefficient matrix $A(\Sigma)$ of the vectorized Lyapunov equation. 

First, recall that when the matrices $M$ and $C$ are given, the continuous Lyapunov equation from~\eqref{eq:lyapunoveq} is uniquely solvable for the symmetric matrix $\Sigma$ if and only if 
no two eigenvalues of $M$ add up to zero.  This well known fact can be shown by vectorizing the equation to
\begin{equation}
\label{eq:lyapunoveq usual vec}
  (I_p \otimes M + M \otimes I_p) \mathrm{vec}(\Sigma) = -\mathrm{vec}(C),
\end{equation}
where $\otimes$ is the Kronecker product and $\mathrm{vec}(\cdot)$ is the columnwise vectorization of a matrix; see, e.g., \cite{bernstein2016}.  The coefficient matrix $I_p \otimes M + M \otimes I_p$ is a Kronecker sum, and it follows that its eigenvalues are the pairwise sums of the eigenvalues of $M$. 
If we now additionally assume that $C$ is positive definite, then  Lyapunov's theorem \citep[Theorem 2.2.1]{horn_johnson_1991} yields  that the Lyapunov equation from~\eqref{eq:lyapunoveq} has a unique positive definite solution $\Sigma$ if and only if $M$ is a stable matrix.

However, solving for $M$ given two symmetric (and in our context positive definite) matrices $\Sigma$ and $C$ is a more difficult question. In general, it is not possible to have a unique solution for $M$ due to the dimensionality problems mentioned in \Cref{lem:dim}.
The graphical perspective of the Lyapunov models motivates considering sparse matrices $M$ and asking the solvability question in a new light, as we illustrated in \Cref{ex:Identifiability3cycle}.

\begin{lemma}
\label{lem:ASigma}
Vectorizing the Lyapunov equation~\eqref{eq:lyapunoveq}, we obtain the system 
\begin{equation}
 \label{eq:ASigmaraw}
    ((\Sigma \otimes I_{p})+(I_{p} \otimes \Sigma) K_p) \mathrm{vec}(M) = -\mathrm{vec}(C),
\end{equation}
where $K_p$ is the $p \times p$ commutation matrix.
\end{lemma}

The commutation matrix $K_p$ is the symmetric permutation matrix that transforms the vectorization of a $p\times p$ matrix to the vectorization of its transpose \citep[p.~54]{neudecker1999}. 
\begin{proof}[Proof of \Cref{lem:ASigma}]
 It holds that
     \begin{multline*}
        \mathrm{vec}(M\Sigma+\Sigma M^{\top})= \mathrm{vec}(M\Sigma)+\mathrm{vec}(\Sigma M^{\top})\\
        =(\Sigma^{\top} \otimes I_{p}) \mathrm{vec}(M)+(I_{p} \otimes \Sigma) \mathrm{vec}(M^{\top})
        =((\Sigma \otimes I_{p})+(I_{p} \otimes \Sigma)K_{p}) \mathrm{vec}(M).
    \end{multline*}
\end{proof}

The Lyapunov equation \eqref{eq:lyapunoveq} is symmetric and therefore $p(p-1)/2$ equations of the equation system \eqref{eq:ASigmaraw} are redundant.

\begin{definition}
\label{def:ASigma}
Given a $p\times p$ symmetric matrix $\Sigma$, we define the $p(p+1)/2\times p^2$ matrix $A(\Sigma)$ by selecting the rows of
\begin{equation*}
    (\Sigma \otimes I_{p})+(I_{p} \otimes \Sigma) K_{p}
\end{equation*}
indexed by pairs $(k,l)$ with $k\le l$.
\end{definition}

Let $\mathrm{vech}(C)=(C_{kl}:k\le l)$ be the half-vectorization of the symmetric matrix $C$.  Then we can write the Lyapunov equation as
\begin{equation*}
    A(\Sigma) \mathrm{vec}(M)=-\mathrm{vech}(C).
\end{equation*}
As noted, we index the rows of $A(\Sigma)$ by pairs $(k,l)$ with $k\le l$.  To index the columns of $A(\Sigma)$ we will use the potential edges $i\to j$, where we recall that the edge $i\to j$ corresponds to the entry $m_{ji}$ of the matrix $M$. 

\Cref{ex:Identifiability3cycle} displayed $A(\Sigma)$ for the case of $p=3$.  In general, we have
\begin{equation}
\label{eq:doubleindex}
   A(\Sigma)_{(k,l),i\to j}= 
   \begin{cases}
   0, \quad &\text{if} \quad j \neq k,l; \\
   \Sigma_{li}, \quad &\text{if} \quad j= k,\, k \neq l;\\
   \Sigma_{ki}, \quad &\text{if} \quad j=l, \, l \neq k;\\
   2\Sigma_{ji}, \quad &\text{if} \quad j=k=l.
   \end{cases}
\end{equation}

Any 
specific graphical continuous Lyapunov model assumes that $M$ has non-zero entries only for pairs $(j,i)$ for which the underlying graph contains the edge $i\to j$.  We are thus led to select a subset of columns of the coefficient matrix $A(\Sigma)$ when studying solvability of the Lyapunov equation.
By the next lemma, generic and global identifiability of a graphical continuous Lyapunov model are equivalent to rank conditions on the relevant submatrix of $A(\Sigma)$.   

\begin{lemma}
\label{lem:rankidentifiability}
Let $G=(V,E)$ be a directed graph with $V=[p]$, and let $C\in\PD_p$.  
Let $A(\Sigma)_{\cdot,E}$ be the submatrix of $A(\Sigma)$ obtained by selecting the columns indexed by the edges in $E$.
Then the model $\mathcal{M}_{G,C}$ is 
\begin{itemize}
    \item[(i)] globally identifiable if and only if $A(\Sigma)_{\cdot, E}$ has full column rank $|E|$ for all $\Sigma \in \mathcal{M}_{G,C}$;
    \item[(ii)] generically identifiable if and only if there exists a matrix $\Sigma \in \mathcal{M}_{G,C}$ such that $A(\Sigma)_{\cdot, E}$ has full column rank $|E|$.
\end{itemize}
If $\mathcal{M}_{G,C}$ is not generically identifiable, then it is non-identifiable.
\end{lemma}
\begin{proof}
Let $M_0\in\mathrm{Stab}(E)$, and let $\Sigma_0=\phi_{G,C}(M_0)$ be the associated covariance matrix.  The fiber $\mathcal{F}_{G,C}(M_0)$ is the set of all matrices $M\in\mathbb{R}^E$ with
\begin{equation}
\label{eq:Asigmaidenti}
    A(\Sigma_0)_{\cdot,E} \,\mathrm{vec}(M)_{E}=-\mathrm{vech}(C),
\end{equation}
where $\mathrm{vec}(M)_{E}$ is the subvector of $\mathrm{vec}(M)$ that comprises the entries indexed by $(j,i)$ with $i\to j\in E$.
Hence, $\mathcal{F}_{G,C}(M_0)=\{M_0\}$ precisely when $A(\Sigma_0)_{\cdot,E}$ has full column rank such that \eqref{eq:Asigmaidenti} has a unique solution.  Claim (i) is now evident. 

To prove (ii), note that  $A(\Sigma)_{\cdot,E}$ has full column rank if and only if the vector of all maximal minors of $A(\Sigma)_{\cdot,E}$ is non-zero.  By \eqref{eq:lyapunoveq usual vec}, the map $\phi_{G,C}$ is a rational map. Consequently, the map taking $M\in\mathrm{Stab}(E)$ to the maximal minors of $A(\phi_{G,C}(M))_{\cdot,E}$ is rational as well.  Now a rational map is non-zero outside a measure zero set if and only if there exists a single point where it is non-zero.  Consequently, the existence of $\Sigma\in\mathcal{M}_{G,C}$ with $A(\Sigma)_{\cdot,E}$ of full column rank implies generic identifiability of $\mathcal{M}_{G,C}$.

Finally, if $\mathcal{M}_{G,C}$ is not generically identifiable then the column rank of  $A(\Sigma_0)_{\cdot,E}$ is strictly smaller than $|E|$ for all $\Sigma_0 = \phi_{G,C}(M_0) \in\mathcal{M}_{G,C}$.    
The fiber $\mathcal{F}_{G,C}(M_0) \subseteq \mathrm{Stab}(E)$ is then the affine 
subspace of solutions to~\eqref{eq:Asigmaidenti} of dimension $\geq 1$.
Hence, $|\mathcal{F}_{G,C}(M_0)|=\infty$ for all $M_0\in\mathrm{Stab}(E)$, and $\mathcal{M}_{G,C}$ is non-identifiable.
\end{proof}

\section{Directed acyclic graphs}\label{sec:DAGs}

In this section, we prove that all models that are given by \emph{directed acyclic graphs} (DAGs) are globally identifiable.  In our setting, a DAG is a directed graph that does not contain any directed cycles other than the always present self-loops $i\to i$, $i\in[p]$.  This case is special in that we are able to make a simple argument based on block structure in the coefficient matrix $A(\Sigma)$. 

By \Cref{prop:subgraph}, in order to prove global identifiability for all DAGs it suffices to treat DAGs that are complete in the sense of the following definition.

\begin{definition}
A directed simple graph $G=(V,E)$ with $V=[p]$ is \emph{complete} if there is an edge between every pair of distinct nodes. 
\end{definition}

A simple graph that also contains all self-loops $i\to i$, $i\in[p]$, is complete if and only if $|E|=p(p+1)/2$. Because vertex relabelling has no impact on identifiability, we can furthermore restrict attention to a single topological ordering.  In other words, it suffices to consider the single complete DAG $G^*$  whose edge set comprises all edges $i\to j$ with $i\ge j$.

\begin{figure}[t]
\begin{center}
\begin{tikzpicture}[->,every node/.style={circle,draw},line width=1pt, node distance=1.5cm]
  \node (1) at (1.5,1.0) {$1$};
  \node (2) at (0,0) {$2$};
  \node (3) at (3,0) {$3$};
\foreach \from/\to in {2/1,3/1,3/2}
\draw (\from) -- (\to);   
\path (1) edge [loop above] (1);
\path (2) edge [loop left] (2);
\path (3) edge [loop right] (3);
 \end{tikzpicture}  
\caption{The complete DAG $G^*$ on 3 nodes.}
\label{fig:DAG3nodes}
\end{center}
\end{figure}
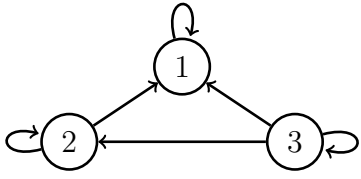

\begin{example}
\label{ex:complete DAG 3 nodes}
Consider the case of $p=3$ nodes, for which the complete DAG $G^*=(V,E^*)$ is shown in \Cref{fig:DAG3nodes}.  The graph encodes the drift matrix
\begin{align*}
M = \begin{pmatrix}
m_{11} & m_{12} & m_{13} \\
0 & m_{22} & m_{23} \\
0 & 0 & m_{33}
\end{pmatrix},
\end{align*}
and the submatrix $A(\Sigma)_{\cdot,E^*}$ is equal to
\begin{align*}
\begin{blockarray}{ccccccc}
& 1\to 1 & 2 \to 1 & 3 \to 1 & 2 \to 2  & 3 \to 2 & 3 \to 3 \\
\begin{block}{c(cccccc)}
            (1,1) & 2 \Sigma_{11}  &2 \Sigma_{12} & 2\Sigma_{13} & 0 & 0 & 0\\
            (1,2) & \Sigma_{12}  & \Sigma_{22}  & \Sigma_{23}  & \Sigma_{12}  & \Sigma_{13} &0\\
            (1,3) & \Sigma_{13} & \Sigma_{23} & \Sigma_{33}  & 0  & 0 & \Sigma_{13}\\
            (2,2) & 0 & 0 & 0  & 2 \Sigma_{22} & 2\Sigma_{23} & 0 \\
            (2,3) & 0 & 0 & 0 & \Sigma_{23}    & \Sigma_{33} & \Sigma_{23}\\
            (3,3) & 0 & 0 & 0 & 0 & 0 & 2 \Sigma_{33}\\
\end{block}
\end{blockarray}.
\end{align*}

Up to some rows being scaled by 2, the three diagonal blocks are principal minors of the positive definite matrix $\Sigma$. Therefore, it holds for all $\Sigma\in\PD_3$ that
\begin{align*}
    \det A(\Sigma)_{\cdot,E^*}&= \begin{vmatrix}
                 2 \Sigma_{11} &2 \Sigma_{12} & 2 \Sigma_{13}\\
              \Sigma_{12}   & \Sigma_{22}  & \Sigma_{23}\\
              \Sigma_{13} & \Sigma_{23} & \Sigma_{33}
    \end{vmatrix} \cdot 
    \begin{vmatrix}
        2\Sigma_{22} & 2 \Sigma_{23}\\
        \Sigma_{23} & \Sigma_{33}
    \end{vmatrix} \cdot
    |2\Sigma_{33}|\\
    &=2^3 \cdot \det(\Sigma)\cdot \det(\Sigma_{\{2,3\},\{2,3\}}) \cdot \Sigma_{33} > 0.
\end{align*}
\end{example}

The block structure found in \Cref{ex:complete DAG 3 nodes} generalizes and gives the main result of this section.

\begin{theorem}
\label{theo:dagidenti}
Let $G=(V,E)$ be a DAG with $V=[p]$.  Then the model $\mathcal{M}_{G,C}$ is globally identifiable for every matrix $C\in\PD_p$.
\end{theorem}
\begin{proof}
As noted above, it suffices to consider the complete DAG $G^*=(V,E^*)$ whose edges are $i\to j$ for $i\ge j$.  

Our proof then applies \Cref{lem:rankidentifiability}, which states that  model $\mathcal{M}_{G^*,C}$ is globally identifiable if and only if $\det(A(\Sigma)_{\cdot,E^*})\not=0$ for all $\Sigma \in \mathcal{M}_{G^*,C}$. 

In what follows, let $\Sigma \in \mathrm{PD}_p$. Partition the edge set as $E^*=E_1^*\cup E_2^*\cup\dots\cup E^*_p$, where $E^*_i = \{j\to i: j\ge i\}$.  Similarly, partition the row index set of $A(\Sigma)$ into the disjoint union of the sets $R_k=\{(k,l):l\ge k\}$, $k=1,\dots,p$.  Inspecting~\eqref{eq:doubleindex}, we see that the submatrix
\[
A(\Sigma)_{R_k,E^*_i} =0 \quad\text{if} \ k >i.
\]

Hence, the matrix $A(\Sigma)$ can be arranged in block upper-triangular form, and
\begin{align*}
    \det\left( A(\Sigma)_{\cdot,E^*}\right) = 
    \prod_{i=1}^p \det \left( A(\Sigma)_{R_i,E^*_i}\right).
    \end{align*}
Note that arrangement of columns and rows is consistent with the ordering in \eqref{eq:col_order}.
Inspecting again~\eqref{eq:doubleindex}, we find that $A(\Sigma)_{R_i,E^*_i}$ is equal to the principal submatrix $P(\Sigma)_{\ge i}:=\Sigma_{\{i,\dots,p\},\{i,\dots,p\}}$ but with the first row of $P(\Sigma)_{\ge i}$ (the one indexed by $i$) being multiplied by 2 in $A(\Sigma)_{R_i,E^*_i}$.
Since all principal minors of a positive definite matrix $\Sigma$ are positive, we obtain that
\begin{align*}
      |\det\big(A(\Sigma)_{\cdot,E^*}\big)|=2^p\prod_{i=1}^{p}\det\big( P(\Sigma)_{\ge i}\big)> 0\quad \text{ for all } \  \Sigma \in\mathrm{PD}_p.
\end{align*}
In particular, $A(\Sigma)_{\cdot,E^*}$ has non-vanishing determinant for all $\Sigma\in\mathcal{M}_{G^*,C}$.
\end{proof}

The proof of \Cref{theo:dagidenti} shows that for any complete DAG $G=(V,E)$ the matrix $A(\Sigma)_{\cdot, E}$ is invertible for all $\Sigma \in \mathrm{PD}_p$. Using this fact, the proof of the theorem reveals more information about Lyapunov models arising from DAGs.

\begin{corollary}
\label{cor:completeModel}
Let $G=(V,E)$ be a DAG with $V=[p]$. Then $\mathcal{M}_{G,C}$ is an algebraic and thus closed subset of $\mathrm{PD}_p$.  If $G$ is complete then $\mathcal{M}_{G,C}=\mathrm{PD}_p$.
\end{corollary}
\begin{proof}
Let $G$ be a complete DAG.
By \Cref{theo:dagidenti}, the square matrix $A(\Sigma)_{\cdot,E}$ has full rank for all $\Sigma\in\PD_p$. Therefore, the solution $\mathrm{vec}(M)$ to the vectorized Lyapunov equation \eqref{eq:Asigmaidenti} exists uniquely for all $\Sigma\in\PD_p$. The resulting drift matrix $M$ has the right support by construction, hence $\mathcal{M}_{G,C}=\PD_p$. 

If $G$ is a non-complete DAG, then we may add edges to obtain a complete DAG $\bar G=(V,\bar E)$.  As  $A(\Sigma)_{\cdot,\bar E}$ has full column rank for all $\Sigma\in\PD_p$ the same is true for $A(\Sigma)_{\cdot,E}$; recall  \Cref{prop:subgraph}. Hence, a matrix $\Sigma\in\PD_p$ is in $\mathcal{M}_{G,C}$ if and only if $\textrm{vech}(C)$ is in the column span of $A(\Sigma)_{\cdot,E}$ if and only if the $(\vert E\vert+1)$-minors of the augmented matrix $\left(A(\Sigma)_{\cdot,E}\ \vert \ \textrm{vech}(C)\right)$ vanish. 
The model $\mathcal{M}_{G,C}$ is thus an algebraic subset:  it is the set of positive definite matrices at which these minors vanish.   
\end{proof}

\section{Sums of squares decompositions and finer rank conditions}\label{sec:cyclic}

Directed cycles break the block-diagonal structure found for DAGs (\Cref{theo:dagidenti}) making it difficult to check rank conditions on $A(\Sigma)$.  In this section we show that small cyclic graphs can nevertheless be handled by applying sums of squares decompositions to certify positivity of subdeterminants.   Moreover, we show that our rank conditions may be placed on a smaller matrix containing a basis for the kernel of $A(\Sigma)$.  

In \Cref{ex:Identifiability3cycle}, we proved global identifiability for the 3-cycle by showing that the key factor $\Sigma_{11}\Sigma_{22}\Sigma_{33}-\Sigma_{12}\Sigma_{13}\Sigma_{23}$ in the determinant of $A(\Sigma)_{\cdot,E}$ is positive on $\PD_3$.  We were able to argue this via the positivity of $2\times 2$ principal minors of $\Sigma$.  However, a direct extension of this approach to cyclic graphs with a larger number of nodes is difficult.  Nevertheless, some headway can be made by exploiting the positive-definiteness of $\Sigma$ via its Cholesky decomposition.   

\begin{figure}[t]
\begin{center}
\begin{tikzpicture}[->,every node/.style={circle,draw},line width=1pt]
  \node[ellipse,draw] (1) at (0,0) {1};
  \node (2) at (0,1.7) {$2$};
  \node (3) at (2,1.7) {$3$};
  \node (4) at (2,0) {$4$}; 
\foreach \from/\to in {1/2,2/3,3/4,4/1,1/3,2/4}
\draw (\from) -- (\to);   
\path (1) edge [loop left] (1);
\path (2) edge [loop left] (2);
\path (3) edge [loop right] (3);
\path (4) edge [loop right] (4);
 \end{tikzpicture}  
\caption{A completion of the 4-cycle.}
\label{fig:4cycle}
\end{center}
\end{figure}
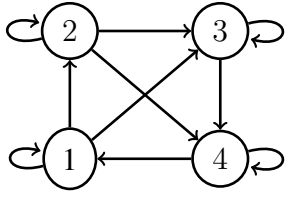

\begin{example}\label{ex:completed4cycle}
Let $G=(V,E)$ be the completion of the 4-cycle with $V=[4]$ and $E= \lbrace 1 \to 1, 2 \to 2, 3 \to 3, 4 \to 4, 1\to 2 ,\,1\to 3,\,2\to 3,2\to 4,\,3\to 4,\,4\to 1 \rbrace$.  It is displayed in \Cref{fig:4cycle}. Let $\Sigma=LL^{\top}$ be the Cholesky decomposition of $\Sigma\in\PD_4$ in terms of the lower-triangular matrix
$$L = \begin{pmatrix}
l_{11} & 0 & 0 & 0 \\
l_{12} & l_{22} & 0 & 0 \\
l_{13} & l_{23} & l_{33} & 0\\
l_{14} & l_{24} & l_{34} & l_{44}
\end{pmatrix}$$ 
with $l_{11},l_{22},l_{33},l_{44}>0$. Then 
\[
\vert\det(A(LL^{\top})_{\cdot,E})\vert=16\,l_{44}^2l_{33}^2l_{22}^4l_{11}^6\cdot |f(L)|,
\]
where the key factor is
\begin{align*}
f(L) \,=\;\:& 
l_{14}^2l_{22}^2l_{33}^2
-l_{12}l_{14}l_{22}l_{24}l_{33}^2
+l_{12}^2l_{24}^2l_{33}^2
+l_{22}^2l_{24}^2l_{33}^2
-l_{13}l_{14}l_{22}^2l_{33}l_{34}\\
&+l_{12}l_{14}l_{22}l_{23}l_{33}l_{34}
+l_{12}l_{13}l_{22}l_{24}l_{33}l_{34}
-l_{12}^2l_{23}l_{24}l_{33}l_{34}
+l_{13}^2l_{22}^2l_{34}^2\\
&-2l_{12}l_{13}l_{22}l_{23}l_{34}^2
+l_{12}^2l_{23}^2l_{34}^2
+l_{12}^2l_{33}^2l_{34}^2
+l_{22}^2l_{33}^2l_{34}^2
+l_{13}^2l_{22}^2l_{44}^2\\
&-2l_{12}l_{13}l_{22}l_{23}l_{44}^2
+l_{12}^2l_{23}^2l_{44}^2
+l_{12}^2l_{33}^2l_{44}^2+l_{22}^2l_{33}^2l_{44}^2.
\end{align*}
A computer algebra system such as \texttt{Macaulay2} with the package from \cite{cifuentes2020} quickly finds a sum of squares (SOS) decomposition for $f$ as
\begin{multline*}
f(L)=\left(\frac{1}{2}l_{14}l_{22}l_{33}-\frac{1}{2}l_{12}l_{24}l_{33}-l_{13}l_{22}l_{34}+l_{12}l_{23}l_{34}\right)^2\\
+\left(-l_{13}l_{22}l_{44}+l_{12}l_{23}l_{44}\right)^2
+\left(l_{12}l_{33}l_{34}\right)^2+\left(l_{12}l_{33}l_{44}\right)^2+\left(l_{22}l_{24}l_{33}\right)^2\\
+\left(l_{22}l_{33}l_{34}\right)^2+\left(l_{22}l_{33}l_{44}\right)^2
+\frac{3}{4}\left(l_{14}l_{22}l_{33}-\frac{1}{3}l_{12}l_{24}l_{33}\right)^2
+\frac{2}{3}\left(l_{12}l_{24}l_{33}\right)^2.
\end{multline*}
Since $l_{22}l_{33}l_{44}>0$, it follows that $f$ is strictly positive for any Cholesky factor $L$.  Therefore, $\vert\det(A(\Sigma)_{\cdot,E})\vert>0$ and we conclude that $\mathcal{M}_{G,C}$ is globally identifiable, no matter the choice of $C\in\PD_4$.
\end{example}

\begin{remark}
A polynomial being a sum of squares is a stronger requirement than the polynomial being non-zero.  Therefore, we could have a non-vanishing determinant even if the considered polynomial factor failed the SOS test.  However, we do not know of an example where this might be the case.
\end{remark}

Observe that $\det(\Sigma)=(\det L)^2=l_{11}^2l_{22}^2l_{33}^2l_{44}^2$ appears as a factor of $\det(A(\Sigma)_{\cdot,E})$ in all our examples so far (recall \Cref{ex:Identifiability3cycle}, \Cref{ex:complete DAG 3 nodes}, and \Cref{ex:completed4cycle}). This phenomenon actually occurs for any complete simple graph (see \Cref{cor:detSfactors} in the 
Appendix) and suggests that identifiability should be encoded in a smaller matrix. Indeed, this information is carried by a specific row restriction of a matrix whose columns form a basis of the kernel of $A(\Sigma)$. 

The kernel of $A(\Sigma)$ is described by the following fact, straightforward to verify; see also
 \cite{barnett1967}.  It parametrizes the stable matrices $M$ that are solutions to the Lyapunov equation in terms of skew-symmetric matrices (matrices $K$ with $K^{\top}=-K$). 
\begin{lemma}
\label{lem:truesignalskew}
Consider the continuous Lyapunov equation from \eqref{eq:lyapunoveq}
for given $\Sigma,C \in \mathrm{PD}_p$.  Then a matrix $M \in \mathbb{R}^{p\times p}$ solves the Lyapunov equation if and only if there exists a skew-symmetric matrix $K$ such that
\begin{align*}
    M=\left(K-\frac{1}{2}C\right)\Sigma^{-1}.
\end{align*}
\end{lemma}
The proof can be found in Section 2 of \cite{barnett1967} and is included here for completeness as Lemma \ref{lem:truesignalskew} plays an important role for the following results.
\begin{proof}
Substituting $M$ into equation \eqref{eq:lyapunoveq} and using the symmetry of $\Sigma$, $C$ and that $K^{\top}=-K$, we obtain 
\begin{align*}
 M\Sigma +\Sigma M^{\top} =& \left(K- \tfrac12 C\right)\Sigma^{-1} \Sigma + \Sigma (\Sigma^{-1})^{\top} \left(K- \tfrac12 C\right)^{\top}\\
=&  \left(K- \tfrac12 C\right) + \left(-K- \tfrac12 C\right) =-C.
\end{align*}
Conversely, since $M$ and $C$ are both symmetric matrices we can write \eqref{eq:lyapunoveq} as 
\begin{align*}
(M \Sigma)^{\top} + \tfrac12 C^{\top}  = - M \Sigma - \tfrac12 C.
\end{align*}
Therefore, the matrix $K=M \Sigma + \tfrac12 C$ is skew-symmetric.
\end{proof}

The space of skew-symmetric matrices has dimension $p(p-1)/2$.  Hence, for $\Sigma \in\mathrm{PD}_p$, the kernel of $A(\Sigma)$ also has dimension $p(p-1)/2$.  We give further details about the spectral properties of $A(\Sigma)$ in \Cref{theo:Asigmaeigenvalues}.
The following result now gives simplified rank conditions for identifiability.

\begin{lemma}
\label{theo:kernelcrit}
Let $G=(V,E)$ be a directed graph with $V=[p]$, and let $C\in\PD_p$. For every $\Sigma\in\PD_p$, let $H(\Sigma)$ be a $p^2\times p(p-1)/2$ matrix whose columns form a basis of the kernel of $A(\Sigma)$, and let $H(\Sigma)_{E^{c},\cdot}$ be the submatrix obtained by restriction to rows corresponding to non-edges $E^c$ of $G$.
Then the associated model $\mathcal{M}_{G,C}$ is 
\begin{itemize}
    \item[(i)] globally identifiable if and only if $H(\Sigma)_{E^{c},\cdot}$ has full column rank $p(p-1)/2$ for all $\Sigma \in \mathcal{M}_{G,C}$;
    \item[(ii)] generically identifiable if and only if there exists a matrix $\Sigma \in \mathcal{M}_{G,C}$ such that $H(\Sigma)_{E^{c},\cdot}$ has full column rank $p(p-1)/2$.
\end{itemize}
\end{lemma}

\begin{proof}
Recall from \Cref{lem:rankidentifiability} that the elements of the fiber are solutions of the equation system \eqref{eq:Asigmaidenti}, which has a unique solution for a given (positive definite) matrix $\Sigma\in\mathcal{M}_{G,C}$ if and only if $A(\Sigma)_{\cdot,E}$ has linearly independent columns. 
The latter condition can be rephrased as follows:
the kernel of $A(\Sigma)$ does not contain any element $\mathrm{vec}(M)\neq 0$ such that $M\in\mathbb{R}^E$.  Put differently, \eqref{eq:Asigmaidenti} admits a unique solution if and only if the column span of $H(\Sigma)$ does not contain any element $\mathrm{vec}(M)\neq 0$ for $M\in\mathbb{R}^E$.  
As $H(\Sigma)$ has linearly independent columns, this latter condition is equivalent to 
the linear independence of the columns of the extended matrix $(H(\Sigma)\mid\mathrm{vec}(M))$ for any non-trivial $M\in\mathbb{R}^E$.  It remains to be proven that this, in turn, is equivalent to the $\vert E^c\vert\times p(p-1)/2$ submatrix $H(\Sigma)_{E^c,\cdot}$ having rank $p(p-1)/2$.

Assume that $H(\Sigma)_{E^c,\cdot}$ has rank $p(p-1)/2$, and consider one of its non-vanishing maximal minors.  This minor can always be extended to a non-vanishing maximal minor of $(H(\Sigma) \mid \mathrm{vec}(M))$ by adding one of the rows corresponding to $m_{ji}\neq 0$. Therefore, the extended matrix has full rank.

For the converse implication, note that if $H(\Sigma)_{E^c,\cdot}$ has rank strictly less than $p(p-1)/2$, then there exists a (not unique) non-trivial $M\in\mathbb{R}^E$ such that $\mathrm{vec}(M)$ belongs to the kernel of $A(\Sigma)$.
\end{proof}

For a convenient choice of a basis of the kernel of $A(\Sigma)$ we may appeal to the following fact.

\begin{lemma}
\label{lem:HSigma}
For a matrix $\Sigma\in\PD_p$, the kernel of $A(\Sigma)$ equals
\begin{align*}
\ker A(\Sigma) &\;=\;\{\mathrm{vec}(K\Sigma^{-1}): K\ \text{skew-symmetric}\}\\
&\;=\; \{\mathrm{vec}(\Sigma K): K\ \text{skew-symmetric}\}.
\end{align*}
\end{lemma}
\begin{proof}
The first equality holds by \Cref{lem:truesignalskew}.  The second equality follows from the fact that $K$ is skew-symmetric if and only if $\Sigma K \Sigma$ is skew-symmetric.
\end{proof}

For $1\le k,l\le p$, let $K^{(k,l)}=e_k\otimes e_l - e_l\otimes e_k$ be the skew-symmetric matrix whose only non-zero entries are 1 in place $(k,l)$ and $-1$ in place $(l,k)$.  Then the set $\{ K^{(k,l)}:k<l\}$ is a basis of the space of $p\times p$ skew-symmetric matrices and, thus, the set $\{ \mathrm{vec}(\Sigma K^{(k,l)}):k<l\}$ is a basis of $\ker A(\Sigma)$.  We may thus choose the matrix $H(\Sigma)$ in \Cref{lem:HSigma} as the matrix with entries
\begin{align}
\label{eq:HSigma}
    H(\Sigma)_{i\to j,(k,l)} = \mathrm{vec}(\Sigma K^{(k,l)})_{ji} = \begin{cases}
    -\Sigma_{lj} &\text{ if } i=k,\\
    \Sigma_{kj} &\text{ if } i=l,\\
    0&\text{ otherwise}.
    \end{cases}
\end{align}
Note that we index the rows of $H(\Sigma)$ by all possible edges of a directed graph (including self-loops), in accordance with the indexing of the columns of $A(\Sigma)$.

\begin{example}\label{ex:kernel} Consider the $6\times 9$ matrix $A(\Sigma)$ in \Cref{ex:Identifiability3cycle} corresponding to $p=3$.  Then the matrix from \eqref{eq:HSigma} is
\begin{align*}
H(\Sigma)=\begin{blockarray}{cccc}
\begin{block}{(ccc)c}
         -\Sigma_{12} & -\Sigma_{13} & 0 & 1 \to 1\\
        -\Sigma_{22} & -\Sigma_{23} & 0 & 1 \to 2\\
        -\Sigma_{23}  & -\Sigma_{33} & 0 & 1 \to 3\\
        \Sigma_{11}  & 0 & -\Sigma_{13} & 2 \to 1\\
        \Sigma_{12}  & 0  & -\Sigma_{23} & 2 \to 2\\
        \Sigma_{13}  & 0 & -\Sigma_{33} & 2 \to 3\\
        0 &  \Sigma_{11} & \Sigma_{12}  & 3 \to 1\\
        0  & \Sigma_{12} & \Sigma_{22} & 3 \to 2\\
        0  & \Sigma_{13} & \Sigma_{23} & 3 \to 3\\
\end{block}
\end{blockarray}.
\end{align*}

Consider the DAG on 3 nodes given in \Cref{fig:DAG3nodes}, for which the set of non-edges is $E^{c}=\lbrace 1\to 2, 1 \to 3, 2 \to 3\rbrace$. Then
\begin{align*}
   \left\vert\det H(\Sigma)_{E^c,\cdot}\right\vert=\left\vert\det
   \begin{pmatrix}
         -\Sigma_{22} & -\Sigma_{23} & 0 \\
        -\Sigma_{23} &  -\Sigma_{33} & 0 \\
         \Sigma_{13} & 0 & -\Sigma_{33} 
  \end{pmatrix}\right\vert
    = \Sigma_{33}(\Sigma_{22}\Sigma_{33}-\Sigma_{23}^2)
\end{align*}
is a product of two principal minors of $\Sigma$, as expected from \Cref{theo:dagidenti}.

Next, let $E^{c}=\lbrace  2 \to 1,  1 \to 3, 3 \to 2 \rbrace$ be the set of non-edges of the 3-cycle from \Cref{fig:cycle}. Then
\begin{align*}
     \left\vert\det H(\Sigma)_{E^c,\cdot}\right\vert=\left\vert \det
    \begin{pmatrix}
        \Sigma_{11}  & 0 & -\Sigma_{13}\\ 
        -\Sigma_{23}  & -\Sigma_{33} & 0\\
         0 & \Sigma_{12} & \Sigma_{22} 
    \end{pmatrix}\right\vert= \Sigma_{11}\Sigma_{22}\Sigma_{33}-\Sigma_{12}\Sigma_{13}\Sigma_{23},
\end{align*}
which is what we obtained in \eqref{eq:detASigma3cycle}.
\end{example}

Following \Cref{ex:completed4cycle}, we can establish global identifiability by computing an SOS decomposition of the determinant of the restricted kernel $H(\Sigma)_{E^c,\cdot}$ using the Cholesky decomposition of $\Sigma$.  
Such computations allowed us to establish:

\begin{proposition}
\label{prop:SOScomputations}
   Let $G=(V,E)$ be a simple graph with $V=[p]$, and let $C\in\PD_p$.  Let $L\in\mathbb{R}^{p\times p}$ be lower-triangular.  If $p\le 4$, then there exists a permutation matrix $P$ such that $\det H(PLL^\top P^\top)_{E^c,\cdot}$ is an everywhere positive sum of squares in the entries of $L$, implying
   that $\mathcal{M}_{G,C}$ is globally identifiable. The same is true for $p=5$ with the possible exception of two types of graphs, depicted in \Cref{fig:probsimplegraphs5}. 
\end{proposition}

For our computer proof of the claims in the proposition, we applied the computer algebra system \texttt{Macaulay2}.  For the graphs in \Cref{fig:probsimplegraphs5}, we additionally employed \texttt{Matlab} toolboxes, but our computations did not terminate.  It is natural to conjecture that \Cref{prop:SOScomputations} holds for all graphs with $p=5$, and even all simple graphs.

\begin{figure}[t]
\begin{center}
\begin{tikzpicture}[->,every node/.style={circle,draw},line width=1pt]
  \node (1) at (0,1.5) [ellipse,draw] {$1$};
  \node (2) at (1.5,2.625) {$2$};
  \node (3) at (3,1.5) {$3$};
  \node (4) at (2.25,0) {$4$}; 
  \node (5) at (0.75,0) {$5$};
\foreach \from/\to in {1/2,1/3,1/4,2/3,2/5,3/4,3/5,4/2,5/1,5/4}
\draw (\from) -- (\to);   
\path (1) edge [loop left] (1);
\path (2) edge [loop above] (2);
\path (3) edge [loop right] (3);
\path (4) edge [loop below] (4);
\path (5) edge [loop below] (5);
\end{tikzpicture}  
\hspace{1cm}
\begin{tikzpicture}[->,every node/.style={circle,draw},line width=1pt]
  \node (1) at (0,1.5) [ellipse,draw] {$1$};
  \node (2) at (1.5,2.625) {$2$};
  \node (3) at (3,1.5) {$3$};
  \node (4) at (2.25,0) {$4$}; 
  \node (5) at (0.75,0) {$5$};
\foreach \from/\to in {1/2,1/3,2/3,2/4,3/4,3/5,4/1,4/5,5/1,5/2}
\draw (\from) -- (\to);   
\path (1) edge [loop left] (1);
\path (2) edge [loop above] (2);
\path (3) edge [loop right] (3);
\path (4) edge [loop below] (4);
\path (5) edge [loop below] (5);
 \end{tikzpicture}  
\end{center}
\caption{The two simple cyclic graphs on 5 nodes for which \Cref{prop:SOScomputations} could not be computationally proved with the techniques employed.}
\label{fig:probsimplegraphs5}
\end{figure}
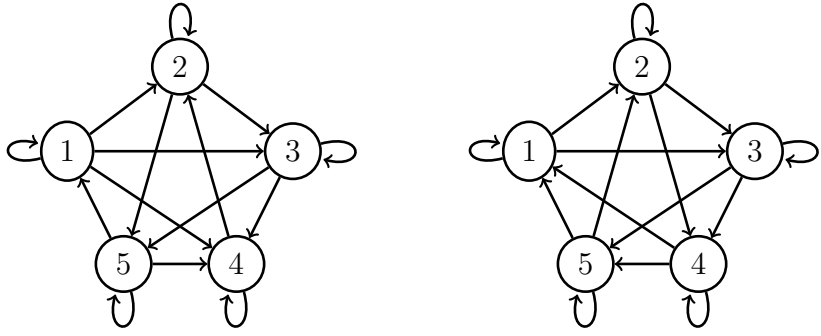

\section{Simple cyclic graphs}
\label{sec:stable}
In this section we 
establish our main result:  global identifiability of all Lyapunov models given by simple cyclic graphs.  Moreover, we can show that simple cyclic graphs give models that are algebraic subsets of the positive definite cone.
Our proofs exploit the parametrization of stable matrices $M$ that are solutions to the Lyapunov equation in terms of skew-symmetric matrices (matrices $K$ with $K^{\top}=-K$); recall \Cref{lem:truesignalskew}.

\begin{theorem}\label{th:simple}
Let $G=(V,E)$ be a directed graph with $V=[p]$.
\begin{itemize}
    \item[(i)] If $G$ is simple, then the model $\mathcal{M}_{G,C}$ is globally identifiable for all $C\in\PD_p$.
    \item[(ii)] If $C\in\PD_p$ is diagonal, then the model $\mathcal{M}_{G,C}$ is globally identifiable if and only if $G$ is simple.
\end{itemize}
\end{theorem}
\begin{proof}
It suffices to prove $(i)$, as $(ii)$ then follows from \Cref{prop:non-simple}.

To prove (i), suppose $G$ is indeed simple.  Let $M_1,M_2 \in \mathrm{Stab}(E)$ be any two matrices that solve the Lyapunov equation \eqref{eq:lyapunoveq} for the same $\Sigma \in \mathcal{M}_{G,C}$. According to \Cref{lem:truesignalskew} there exist two skew-symmetric matrices $K_{1}$ and $K_{2}$ such that $M_{1}=(K_{1}-\frac{1}{2}C)\Sigma^{-1}$ and $M_{2}=(K_{2}-\frac{1}{2}C)\Sigma^{-1}$. For the difference we obtain
\begin{align*}
    M:=M_{1}-M_{2} = (K_{1}-\tfrac{1}{2}C)\Sigma^{-1}-(K_{2}-\tfrac{1}{2}C)\Sigma^{-1} =(K_{1}-K_{2})\Sigma^{-1}.
\end{align*}
The difference $K=K_{1}-K_{2}$ is again skew-symmetric, so that $M$ is the product of a skew-symmetric matrix $K$ and the positive definite matrix $\Sigma^{-1}$. 

As $\Sigma$ and $\Sigma^{-1}$ are positive definite, we may form the square roots $\Sigma^{\frac{1}{2}}$ and $\Sigma^{-\frac{1}{2}}$. Observe that $M=K \Sigma^{-1}$ is similar to $\tilde{M}=\Sigma^{-\frac{1}{2}} K \Sigma^{-1} \Sigma^{\frac{1}{2}}$, which is skew-symmetric since  
\begin{equation*}
\label{eq:skewsimilarity}
    \tilde{M}^{\top} = (\Sigma^{-\frac{1}{2}} K \Sigma^{-\frac{1}{2}})^{\top} = \Sigma^{-\frac{1}{2}} K^{\top} \Sigma^{-\frac{1}{2}}=-\Sigma^{-\frac{1}{2}} K \Sigma^{-\frac{1}{2}} = - \tilde{M}.
\end{equation*} 
The nonzero eigenvalues of real skew-symmetric matrices are purely imaginary. Let $i \lambda_{1},\dots,i\lambda_{p}$ with $\lambda_{i} \in \mathbb{R}$ be the eigenvalues of $\tilde{M}$. Then the similarity implies that $M$ has the same eigenvalues.

Observe now that the eigenvalues of $M^2$ are $-\lambda_{1}^2,\dots,-\lambda_{p}^2$ and thus $\mathrm{tr}(M^2) \leq 0$.  
As $M$ is supported over a simple graph, it holds for all pairs of indices $i \neq j$ that $m_{ij} m_{ji} = 0$. Hence, the diagonal of $M^2$ is given by the squared diagonal elements of $M$, i.e., $(M^2)_{ii}=m_{ii}^2$. 
It follows that
\begin{align*}
    0\leq  \sum_{i=1}^p m_{ii}^2 = \mathrm{tr}(M^2) =\sum_{i=1}^{p} -\lambda_{i}^2 \leq 0, 
\end{align*}
which implies that $\lambda_{i}^2=0$ for all $i = 1,\dots,p$. But this is only true if $\lambda_{i}=0$ for all $i = 1,\dots,p$. Therefore, all eigenvalues of $M$ are zero. Using the similarity of $M$ with $\tilde{M}$ and that skew-symmetric matrices are diagonalizable, we deduce that $M$ is similar to the zero matrix.  But then $M=0$ and consequently $M_1 = M_2$, which shows that the Lyapunov equation admits a unique 
solution in $\mathrm{Stab}(E)$.
\end{proof}

In addition to global identifiability, we have a generalization of \Cref{cor:completeModel} to general simple graphs.

\begin{corollary}
\label{conj:completeModel}
Let $G=(V,E)$ be a simple graph with $V=[p]$. Then $\mathcal{M}_{G,C}$ is an algebraic and thus closed subset of $\mathrm{PD}_p$.  If $G$ is complete then $\mathcal{M}_{G,C}=\mathrm{PD}_p$.
\end{corollary}
\begin{proof}
Consider first the case where $G$ is complete (with an edge between every pair of nodes).  Let $\Sigma_0\in\PD_p$ be an arbitrary positive definite matrix.  Choosing $M=-I_p$, the negated identity matrix, shows that $\Sigma_0$ belongs to the model $\mathcal{M}_{G,C_0}$ for $C_0=2\Sigma_0$.  By \Cref{th:simple} and \Cref{lem:rankidentifiability}, we obtain that the determinant of $A(\Sigma)_{\cdot,E}$ is non-zero at every matrix in $\mathcal{M}_{G,C_0}$ and, in particular, at $\Sigma_0$.  We conclude that $\det(A(\Sigma)_{\cdot,E})\not=0$ on all of $\PD_p$.  As in the proof of \Cref{cor:completeModel}, we deduce that $\mathcal{M}_{G,C}=\PD_p$ for all $C\in\PD_p$.

If $G$ is not complete, then it can be augmented to a complete graph $\bar G=(V,\bar E)$, and we may complete the proof in analogy to the proof of \Cref{cor:completeModel}.
\end{proof}

\section{Non-simple graphs}\label{sec:nonsimple}

In this section, we consider directed graphs $G=(V,E)$ that are allowed to be non-simple, i.e., may contain a two-cycle.  In our study, we restrict attention to the case where $C\in\PD_p$ is diagonal.    
\Cref{prop:non-simple} tells us that, for $C$ diagonal, a model $\mathcal{M}_{G,C}$ given by a non-simple graph $G$ can never be globally identifiable.  However,  non-simple graphs with at most $p(p+1)/2$ edges may still give generically identifiable models (\Cref{def:notions_identifiability}, \Cref{lem:dim}).  We are able to provide a combinatorial condition that is necessary for generic identifiability, and we computationally classify all graphs with $p\le 5$ nodes.  Our study reveals examples for which generic identifiability depends in subtle ways on the pattern of edges.

\begin{figure}[t]
    \centering
   \begin{tikzpicture}[->,every node/.style={circle,draw},line width=1pt, node distance=1.5cm]
  \node (1) at (0,0)     {$1$};
  \node (2) at (2,0)     {$2$};
  \node (3) at (4,0)     {$3$};

\foreach \from/\to in {2/3}
\draw (\from) -- (\to);  
\draw (1) to [bend right]  (2);
\draw (2) to [bend right] (1);
\path (1) edge [loop left] (1);
\path (2) edge [loop above] (2);
\path (3) edge [loop right] (3);
\end{tikzpicture}
    \caption{Non-simple graph on 3 nodes.}
    \label{fig:twocycleplus1}
\end{figure}
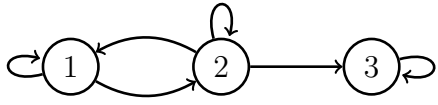

We begin with a small example.
\begin{example}\label{ex:generic}
Let $G=(V,E)$ be the graph from \Cref{fig:twocycleplus1}, a 2-cycle with an additional edge pointing to a third node, 
and let $C\in\PD_3$ be a diagonal matrix.
To inspect identifiability of $\mathcal{M}_{G,C}$, we may use the kernel basis of \Cref{ex:kernel} with the set of non-edges $E^c=\lbrace 1 \to 3, 3 \to 1,3 \to 2 \rbrace$. We find
\begin{align*}
   \det H(\Sigma)_{E^c,\cdot}=\det
    \begin{pmatrix}
         -\Sigma_{23} &-\Sigma_{33} & 0\\
        0 & \Sigma_{11} & \Sigma_{12} \\
        0 &  \Sigma_{12}& \Sigma_{22} 
  \end{pmatrix}
    = \Sigma_{23}\left(\Sigma_{12}^2-\Sigma_{11}\Sigma_{22}\right).
\end{align*}
Since $\mathcal{M}_{G,C}$ contains positive definite matrices with both vanishing and non-vanishing $\Sigma_{23}$, we conclude that $\mathcal{M}_{G,C}$ is generically (but not globally) identifiable.

Note that the matrices $\Sigma\in\mathcal{M}_{G,C}$ with $\Sigma_{23}=0$ are obtained precisely from the drift matrices in the lower-dimensional set $\lbrace M \in \mathrm{Stab}(E): m_{32}=0 \rbrace$.   Indeed, if $m_{32}=0$, then the situation is as if the $2 \to 3$ edge were removed, and we will see in \Cref{prop:trekcrit} that this implies $\Sigma_{23}=0$ when $C$ is diagonal.   Conversely, when solving for $\Sigma$ given a drift matrix $M\in\mathbb{R}^E$ we find that $\Sigma_{23}$ is a rational function of $(M,C)$ whose numerator is
\[
m_{32}\left( c_{11}m_{21}^2\mathrm{tr}(M) + c_{22} m_{11}^2\mathrm{tr}(M) +c_{22} \det(M)\right).
\]
As $C$ is positive definite and $M$ stable, the second factor is negative.  Thus, if $\Sigma=\Sigma(M,C)$ is a positive definite matrix in $\mathcal{M}_{G,C}$, then $\Sigma_{23}=0$ implies $m_{32}=0$.
\end{example}

By \Cref{lem:dim}, $|E|\le p(p+1)/2$ is a necessary condition for generic identifiability of the model of a graph $G=(V,E)$.  We now show how this bound may be improved by accounting for knowledge about vanishing covariances. 

\begin{definition}
A \emph{trek} is a sequence of edges of the form
\begin{equation*}
  l_m \leftarrow l_{m-1} \leftarrow \dots \leftarrow l_1 \leftarrow t \rightarrow r_1 \rightarrow \dots \rightarrow r_{n-1} \rightarrow r_{n}.
\end{equation*}
The node $t$ is the top node of the trek. The directed paths $l_m \leftarrow l_{m-1} \leftarrow \dots \leftarrow l_1$ and $r_1 \rightarrow \dots \rightarrow r_{n-1} \rightarrow r_{n}$ are the left and the right side of the trek, respectively.  The definition allows for one or both sides to be trivial, so  directed paths and also single nodes are treks. 
\end{definition}

From \citet[Corollary 2.3]{hansen2020}, we deduce the following fact.

\begin{proposition}
\label{prop:trekcrit}
Let $G=(V,E)$ be a directed graph with $V=[p]$, and let $C\in\PD_p$ be diagonal. If there is no trek from $i$ to $j$ in $G$, then $\Sigma_{ij}=0$ in all matrices $\Sigma \in \mathcal{M}_{G,C}$.
\end{proposition}

\begin{example}
\label{ex:subgraphnotiden}
Let $C\in\PD_4$ be diagonal.  Then the left graph $G_1=(V,E_1)$ in \Cref{fig:intro} defines a generically identifiable model but its subgraph $G_2=(V,E_2)$ does not.  This example stresses that global identifiability is needed in \Cref{prop:subgraph}.  But why is $\mathcal{M}_{G_2,C}$ non-identifiable despite $G_2$ having fewer edges?  We observe that $G_2$ contains no trek between 2 and 4 and no trek between 3 and 4. \Cref{prop:trekcrit} yields $\Sigma_{24}=\Sigma_{34}=0$. Although the $\PD_4$-cone has dimension $\binom{4+1}{2}=10$, the existence of the constraints $\Sigma_{24}=\Sigma_{34}=0$ implies that $\dim(\mathcal{M}_{G_2,C})\leq 10-2=8$. Since $\vert E_2\vert =9>8$, non-identifiability  follows from by \Cref{lem:dim}.

As a last subtlety, we emphasize that if we remove one of the edges $2 \to 1$, $3 \to 1$, or $4\to 1$ of $G_2$, we are left again with a generically identifiable model. 
\end{example}

\begin{figure}[t]
\centering
\begin{minipage}{.4\textwidth}
        \begin{tikzpicture}[->,every node/.style={circle,draw},line width=1pt, node distance=1.5cm]
  \node (1) at (2,1)     {$1$};
  \node (2) at (4,2)     {$2$};
  \node (3) at (4,0)     {$3$};
  \node (4) at (0,1)     {$4$};

\foreach \from/\to in {2/1,3/1}
\draw (\from) -- (\to);  
\draw (2) to [bend right]  (3);
\draw (3) to [bend right] (2);
\draw (4) to [bend right] (1);
\draw[color=red] (1) to [bend right] (4);
\path (1) edge [loop above] (1);
\path (2) edge [loop right] (2);
\path (3) edge [loop right] (3);
\path (4) edge [loop above] (4);
\end{tikzpicture} 
\end{minipage}\qquad\qquad
\begin{minipage}{.4\textwidth}
       \begin{tikzpicture}[->,every node/.style={circle,draw},line width=1pt, node distance=1.5cm]
  \node (1) at (2,1)     {$1$};
  \node (2) at (4,2)     {$2$};
  \node (3) at (4,0)     {$3$};
  \node (4) at (0,1)     {$4$};

\foreach \from/\to in {4/1,2/1,3/1}
\draw (\from) -- (\to);  
\draw (2) to [bend right]  (3);
\draw (3) to [bend right] (2);
\path (1) edge [loop above] (1);
\path (2) edge [loop right] (2);
\path (3) edge [loop right] (3);
\path (4) edge [loop above] (4);
\end{tikzpicture} 
\end{minipage}
\caption{Left: graph $G_1$ on 4 nodes with $\mathcal{M}_{G_1,C}$ generically identifiable. Right: subgraph $G_2$ of $G_1$ such that $\mathcal{M}_{G_2,C}$ is non-identifiable. $C\in\PD_4$ is diagonal.}
\label{fig:intro}
\end{figure}
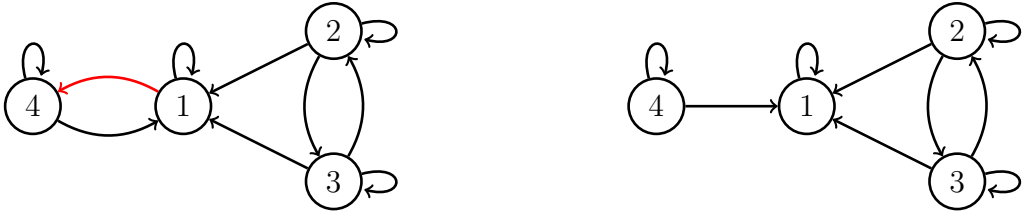

The ideas in \Cref{ex:subgraphnotiden} can be generalized into a sharper necessary condition for identifiability that is a consequence of \Cref{lem:dim} and \Cref{prop:trekcrit}.  

\begin{corollary}
\label{lem:neccritident}
Let $G=(V,E)$ be a directed graph with $V=[p]$.  If $\mathcal{M}_{G,C}$ is generically identifiable for a diagonal matrix $C\in\PD_p$, then it has to hold that 
\begin{equation}
\label{eq:neccriteq}
    |E| \;\leq\; \frac{p(p+1)}{2} - \#\big\lbrace \,\{i,j\}: i,j \in V \text{ with no trek between them}\, \big\rbrace. 
\end{equation}
\end{corollary}

With this criterion in hand, we can construct graphs of arbitrary size $p$ and fewer than $p(p+1)/2$ edges that yield non-identifiable models. 

\begin{corollary}
Consider the graph $G=(V,E)$ with $p\geq 4$ nodes displayed in \Cref{fig:familyfig}. The model $\mathcal{M}_{G,C}$ is non-identifiable for any diagonal $C\in\PD_p$. 
\end{corollary}
 
\begin{figure}[t]
\centering
\begin{tikzpicture}[->,every node/.style={circle,draw},line width=1pt, node distance=1.5cm]
  \node (1) at (2,1)     {$1$};
  \node (2) at (4,2)     {$2$};
  \node (3) at (4,0)     {$3$};
  \node (4) at (0,-1)     {$4$};
  \node (p) at (0,3)     {$p$};
  \node (5) at (0,0)     {$5$};

\foreach \from/\to in {4/1,2/1,3/1,p/1,5/1}
\draw (\from) -- (\to);  
\draw (2) to [bend right]  (3);
\draw (3) to [bend right] (2);
\draw[dotted] (0,1) to [] (1);
\draw[dotted] (0,2) to [] (1);
\path (1) edge [loop above] (1);
\path (2) edge [loop right] (2);
\path (3) edge [loop right] (3);
\path (4) edge [loop left] (4);
\path (5) edge [loop left] (5);
\path (p) edge [loop left] (p);
\end{tikzpicture} 
\caption{Graph $G$ with $V=[p]$ such that $\mathcal{M}_{G,C}$ is non-identifiable for diagonal $C\in\PD_p$.}
\label{fig:familyfig}
\end{figure}
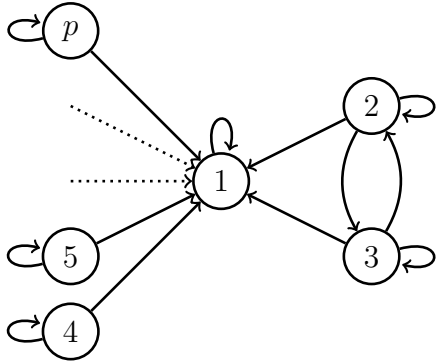

\begin{proof}
The number of parameters $\vert E\vert$ is
\begin{align*}
    2 \, \text{ (edges from 2-cycle) } &+\,\, p-1 \text{ (edges pointing to node 1) } \\
    +\,\, &p \text{ (parameters due to the selfloops)} = \, \, 2p+1.
\end{align*}
There are no treks between any pair of nodes $\lbrace 2,\dots,p \rbrace$ except for the pair $(2,3)$. This results in $\binom{p-1}{2}-1$ (unordered) pairs of nodes with no trek. \Cref{lem:neccritident} implies that 
\begin{align*}
&\dim(\mathcal{M}_{G,C})\leq \frac{p(p+1)}{2}-\binom{p-1}{2}+1 =2p.
\end{align*}
\end{proof}

Unfortunately, the criterion in \Cref{lem:neccritident} is not sufficient.

\begin{figure}[t]
\centering
\begin{minipage}{.4\textwidth}

 \begin{tikzpicture}[->,every node/.style={circle,draw},line width=1pt, node distance=1.5cm]
  \node (1) at (0,.8)     {$1$};
  \node (2) at (2,1.6)     {$2$};
  \node (3) at (2,0)     {$3$};
  \node (4) at (4,0.8)     {$4$};

\foreach \from/\to in {2/1,3/1,2/4,3/4}
\draw (\from) -- (\to);  
\draw (2) to [bend right]  (3);
\draw (3) to [bend right] (2);
\path (1) edge [loop above] (1);
\path (2) edge [loop above] (2);
\path (3) edge [loop below] (3);
\path (4) edge [loop above] (4);
\end{tikzpicture} 
\end{minipage}\hspace{1cm}
\begin{minipage}{.4\textwidth}
 \begin{tikzpicture}[->,every node/.style={circle,draw},line width=1pt, node distance=1.5cm]
  \node (1) at (0,.8)     {$1$};
  \node (2) at (2,1.6)     {$2$};
  \node (3) at (2,0)     {$3$};
  \node (4) at (4,0.8)     {$4$};

\foreach \from/\to in {1/2,1/3,4/2,4/3}
\draw (\from) -- (\to);  
\draw (2) to [bend right]  (3);
\draw (3) to [bend right] (2);
\path (1) edge [loop above] (1);
\path (2) edge [loop above] (2);
\path (3) edge [loop below] (3);
\path (4) edge [loop above] (4);
\end{tikzpicture} 
\end{minipage}
\caption{Left: graph fulfilling the criterion in \Cref{lem:neccritident}, yet yields a non-identifiable model. Right: Reversing edges retains non-identifiability, due to \Cref{lem:neccritident}, as $\Sigma_{14}=0$.}
\label{fig:necgraphnotiden}
\end{figure}
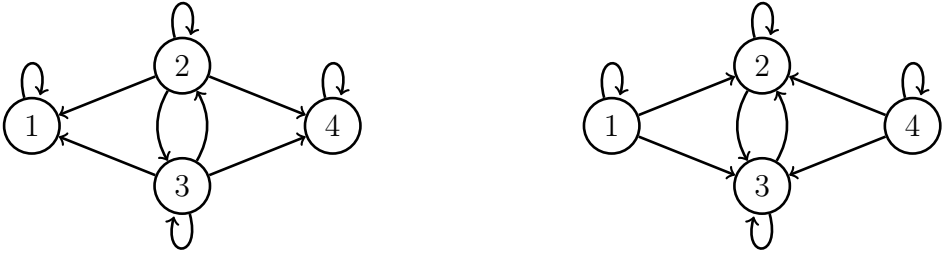

\begin{example}
\label{ex:diamnotid}
Let $G_{1}=(V,E)$ be the left graph in \Cref{fig:necgraphnotiden}.
Graph $G_1$ fulfills the necessary condition of \Cref{lem:neccritident} as the number of parameters is $6+4=10$ and all pairs of nodes are connected with a trek, which is why the right side of equation \eqref{eq:neccriteq} is also $\binom{4+1}{2}=10$. However, $A(\Sigma)_{\cdot,E} \in \mathbb{R}^{10 \times 10}$ does not have full rank because the columns of $A(\Sigma)$ may be linearly combined to 
\begin{align*}
    &\Sigma_{13} A(\Sigma)_{\cdot,2 \to 1} +\Sigma_{23} A(\Sigma)_{\cdot,2 \to 2}+ \Sigma_{33} A(\Sigma)_{\cdot, 2 \to 3}+\Sigma_{34} A(\Sigma)_{\cdot,2 \to 4}\\
    -&\Sigma_{12} A(\Sigma)_{\cdot, 3 \to 1} -\Sigma_{22} A(\Sigma)_{\cdot, 3 \to 2} - \Sigma_{23} A(\Sigma)_{\cdot, 3 \to 3} -\Sigma_{24} A(\Sigma)_{\cdot, 3 \to 4}=0.
\end{align*}
Therefore, the model $\mathcal{M}_{G_1,C}$ is non-identifiable for $C$ diagonal despite fulfilling the necessary criterion. The right graph in \Cref{fig:necgraphnotiden} yields a  non-identifiable model for the simple reason that the necessary condition of \Cref{lem:neccritident} is violated due to the absence of a trek between nodes $1$ and $4$. 
\end{example}

For smaller examples, we may check generic identifiability by choosing random drift matrices and determining whether the resulting matrix $\Sigma$ satisfies the rank condition from \Cref{lem:rankidentifiability}.  When this does not succeed we can check symbolically whether the corresponding restriction of the coefficient matrix $A(\Sigma)$ or the restricted kernel basis $H(\Sigma)$ from \Cref{theo:kernelcrit} is rank-deficient, thus implying non-identifiability.  
We implemented this strategy for all non-simple graphs with $p\le 5$ nodes and less than $p(p+1)/2$ parameters.  As justified by \Cref{prop:Cdiag}, we took $C=I_p$ in our computations.  This led to the results displayed in \Cref{tab:nonsimpletable}, which shows that the majority of graphs are generically identifiable. The details of the computations can be found at \url{https://mathrepo.mis.mpg.de/LyapunovIdentifiability}.

\begin{table}[tbhp]
{\footnotesize
  \caption{Classification of models with $p=3,4,5$ nodes and $C=I_p$.  The last
column displays the number of non-identifiable models whose underlying graphs satisfy the necessary criterion for generic identifiability in \Cref{lem:neccritident}.}
\label{tab:nonsimpletable}
}
\begin{center}
\begin{tabular}{||c |c |c |c||} 
 \hline
 nodes  & total non-simple  & non-identifiable &   non-identifiable satisfying \eqref{eq:neccriteq}\\ [0.5ex] 
 \hline\hline
 3 & 2 & 0 & 0 \\ 
 \hline
 4 & 80 & 3 & 2\\
 \hline
 5 & 4862 & 68 & 37 \\ 
 \hline
\end{tabular}
\end{center}
\end{table}

\section{Conclusion}
\label{sec:conclusion}

Graphical continuous Lyapunov models offer a new perspective on modeling the covariance structure of multivariate data by relating each observation to an underlying continuous-time dynamic process.  The resulting covariance structure is determined by the continuous Lyapunov equation.  Our work addresses the fundamental problem of whether, up to joint scaling, the parameters of the dynamic process can be identified from the covariance matrix of the cross-sectional equilibrium observations.  Our main contribution shows that simple graphs yield globally identifiable models, and that the graph being simple is necessary and sufficient for  global identifiability in the case where the volatility matrix $C$ is diagonal.  Moreover, we are able to show that the models of simple graphs are closed algebraic subsets of the positive definite cone.  In particular, the models of complete simple graphs equal the entire positive definite cone.  

Our analysis of directed acyclic graphs (DAGs) highlights block structure in the coefficient matrix for the Lyapunov equation.  This leads to a straightforward proof of global identifiability and also reveals that the determinant studied in our rank conditions is a positive sum of squares in the entries of a Cholesky factor.  This sum of squares property was also observed in small cyclic graphs.

While we were able to characterize global identifiability, we know less about  generic identifiability of graphical Lyapunov models.  Our results include an effective necessary but not sufficient graphical criterion for non-simple graphs to be generically identifiable.  We also obtain a computational classification of graphs with up to 5 nodes, and we hope that future research will lead to an improved understanding of generic identifiability of the models we considered.

\renewcommand{\bibfont}{\small}
\renewcommand\refname{\centerline{\sc References}}
\bibliography{lyapunov}

\begin{thebibliography}{24}
\providecommand{\natexlab}[1]{#1}
\providecommand{\url}[1]{\texttt{#1}}
\expandafter\ifx\csname urlstyle\endcsname\relax
  \providecommand{\doi}[1]{doi: #1}\else
  \providecommand{\doi}{doi: \begingroup \urlstyle{rm}\Url}\fi

\bibitem[Barber et~al.(2022)Barber, Drton, Sturma, and Weihs]{BDSW22}
Rina~Foygel Barber, Mathias Drton, Nils Sturma, and Luca Weihs.
\newblock Half-trek criterion for identifiability of latent variable models.
\newblock \emph{Ann. Statist.}, 50\penalty0 (6):\penalty0 3174--3196, 2022.

\bibitem[Barnett and Storey(1967)]{barnett1967}
S.~Barnett and C.~Storey.
\newblock Analysis and synthesis of stability matrices.
\newblock \emph{J. Differential Equations}, 3:\penalty0 414--422, 1967.

\bibitem[Bernstein(2011)]{bernstein2016}
Dennis~S. Bernstein.
\newblock \emph{Matrix Mathematics: Theory, Facts, and Formulas}.
\newblock Princeton University Press, second edition, 2011.

\bibitem[Bhaya et~al.(2003)Bhaya, Kaszkurewicz, and Santos]{bhaya2003}
Amit Bhaya, Eugenius Kaszkurewicz, and R.~Santos.
\newblock Characterizations of classes of stable matrices.
\newblock \emph{Linear Algebra and Its Applications}, 374:\penalty0 159--174, 11 2003.

\bibitem[Brito and Pearl(2006)]{brito2006graphical}
Carlos Brito and Judea Pearl.
\newblock Graphical condition for identification in recursive sem.
\newblock In \emph{Proceedings of the Twenty-Second Conference on Uncertainty in Artificial Intelligence}, pages 47--54, 2006.

\bibitem[Cifuentes et~al.(2020)Cifuentes, Kahle, and Parrilo]{cifuentes2020}
Diego Cifuentes, Thomas Kahle, and Pablo Parrilo.
\newblock Sums of squares in macaulay2.
\newblock \emph{Journal of Software for Algebra and Geometry}, 10:\penalty0 17--24, 03 2020.

\bibitem[Dettling et~al.(2022)Dettling, Drton, and Kolar]{Dettling2022}
Philipp Dettling, Mathias Drton, and Mladen Kolar.
\newblock On the lasso for graphical continuous lyapunov models, 2022.
\newblock URL \url{https://arxiv.org/abs/2208.13572}.

\bibitem[Drton(2018)]{drton2018}
Mathias Drton.
\newblock Algebraic problems in structural equation modeling.
\newblock In \emph{The 50th anniversary of {G}r\"{o}bner bases}, volume~77 of \emph{Adv. Stud. Pure Math.}, pages 35--86. Math. Soc. Japan, Tokyo, 2018.

\bibitem[Drton and Weihs(2016)]{drton2016generic}
Mathias Drton and Luca Weihs.
\newblock Generic identifiability of linear structural equation models by ancestor decomposition.
\newblock \emph{Scandinavian Journal of Statistics}, 43\penalty0 (4):\penalty0 1035--1045, 2016.

\bibitem[Drton et~al.(2011)Drton, Foygel, and Sullivant]{Drton2011}
Mathias Drton, Rina Foygel, and Seth Sullivant.
\newblock Global identifiability of linear structural equation models.
\newblock \emph{Ann. Statist.}, 39\penalty0 (2):\penalty0 865--886, 2011.

\bibitem[Fitch(2019)]{katie2019}
Katherine~E. Fitch.
\newblock Learning directed graphical models from {G}aussian data.
\newblock \emph{arXiv}, abs/1906.08050, 2019.

\bibitem[Foygel et~al.(2012)Foygel, Draisma, and Drton]{foygel2012half}
Rina Foygel, Jan Draisma, and Mathias Drton.
\newblock Half-trek criterion for generic identifiability of linear structural equation models.
\newblock \emph{Ann. Statist.}, 40\penalty0 (3):\penalty0 1682--1713, 2012.

\bibitem[Horn and Johnson(1991)]{horn_johnson_1991}
Roger~A. Horn and Charles~R. Johnson.
\newblock \emph{Topics in Matrix Analysis}.
\newblock Cambridge University Press, 1991.

\bibitem[Kumor et~al.(2019)Kumor, Chen, and Bareinboim]{kumor2019}
Daniel Kumor, Bryant Chen, and Elias Bareinboim.
\newblock Efficient identification in linear structural causal models with instrumental cutsets.
\newblock In \emph{Advances in Neural Information Processing Systems}, volume~32. Curran Associates, Inc., 2019.

\bibitem[Maathuis et~al.(2019)Maathuis, Drton, Lauritzen, and Wainwright]{handbook:graphical:models}
Marloes Maathuis, Mathias Drton, Steffen Lauritzen, and Martin Wainwright, editors.
\newblock \emph{Handbook of Graphical Models}.
\newblock Chapman \& Hall/CRC Handbooks of Modern Statistical Methods. CRC Press, Boca Raton, FL, 2019.

\bibitem[Magnus and Neudecker(1999)]{neudecker1999}
Jan~R. Magnus and Heinz Neudecker.
\newblock \emph{Matrix Differential Calculus with Applications in Statistics and Econometrics}.
\newblock Wiley Series in Probability and Statistics. John Wiley \& Sons, Ltd., Chichester, 1999.

\bibitem[Pearl(2009)]{pearl:2009}
Judea Pearl.
\newblock \emph{Causality}.
\newblock Cambridge University Press, Cambridge, second edition, 2009.
\newblock Models, Reasoning, and Inference.

\bibitem[Peters et~al.(2017)Peters, Janzing, and Sch\"{o}lkopf]{peters2017}
Jonas Peters, Dominik Janzing, and Bernhard Sch\"{o}lkopf.
\newblock \emph{Elements of Causal Inference}.
\newblock Adaptive Computation and Machine Learning. MIT Press, Cambridge, MA, 2017.
\newblock Foundations and learning algorithms.

\bibitem[Risken(1996)]{risken1996}
Hannes Risken.
\newblock \emph{Fokker-planck equation}.
\newblock Springer, 1996.

\bibitem[Sachs et~al.(2005)Sachs, Perez, Pe{\textquoteright}er, Lauffenburger, and Nolan]{Sachs2005}
Karen Sachs, Omar Perez, Dana Pe{\textquoteright}er, Douglas~A. Lauffenburger, and Garry~P. Nolan.
\newblock Causal protein-signaling networks derived from multiparameter single-cell data.
\newblock \emph{Science}, 308\penalty0 (5721):\penalty0 523--529, 2005.

\bibitem[Spirtes et~al.(2000)Spirtes, Glymour, and Scheines]{spirtes:2000}
Peter Spirtes, Clark Glymour, and Richard Scheines.
\newblock \emph{Causation, Prediction, and Search}.
\newblock MIT Press, Cambridge, MA, second edition, 2000.

\bibitem[Sullivant(2018)]{Sullivant2018}
Seth Sullivant.
\newblock \emph{Algebraic statistics}, volume 194 of \emph{Graduate Studies in Mathematics}.
\newblock American Mathematical Society, Providence, RI, 2018.

\bibitem[Varando and Hansen(2020)]{hansen2020}
Gherardo Varando and Niels~Richard Hansen.
\newblock Graphical continuous {L}yapunov models.
\newblock In \emph{Proceedings of the 36th Conference on Uncertainty in Artificial Intelligence}, pages 989--998, 2020.

\bibitem[Young et~al.(2019)Young, Yeung, and Raftery]{Young:2019}
William~Chad Young, Ka~Yee Yeung, and Adrian~E Raftery.
\newblock Identifying dynamical time series model parameters from equilibrium samples, with application to gene regulatory networks.
\newblock \emph{Statistical Modelling}, 19\penalty0 (4), 2019.

\end{thebibliography}

\newpage

\appendix

\section{Volatility matrix: diagonal vs. non-diagonal PD matrix}\label{sec:C}
This section aims at providing insight into the need of the diagonality constraint on the volatility matrix $C\in\PD_p$ of the Lyapunov equation to ensure that some of the stronger results of the paper hold. 

\begin{example}\label{ex:failureTrek}
Let $G$ be the 2-cycle with an additional third node, so $V=[3]$ and $E=\{1\to 1,1\to 2,2\to 1,2\to 2,3\to 3\}$, which encodes drift matrices
\[
M=\begin{pmatrix}
     m_{11} & m_{12} & 0\\
     m_{21} & m_{22} & 0\\
     0 & 0 & m_{33}
\end{pmatrix}.
\]
Let $C=(c_{ij})\in\PD_3$.  Clearly, the graph does not contain any treks between nodes 1 and 3, nor between nodes 2 and 3.   However, a matrix $\Sigma=\phi_{G,C}(M)$ has
\begin{align*}
    \Sigma_{13} &= \frac{c_{23} m_{12}-c_{13} (m_{22}+m_{33})}{(m_{11}m_{22}-m_{12} m_{21})+m_{33}(m_{11}+m_{22}+m_{33})},
\end{align*}
with a denominator that is positive on $\mathrm{Stab}(E)$ and a numerator that is constant zero only if $c_{13}=c_{23}=0$.  The same holds for $\Sigma_{23}$ by symmetry.  This example serves to highlight that \Cref{prop:trekcrit} may be false when $C$ is not diagonal.  Indeed, the treks would need to be allowed to move along new edges that reflect presence of non-zero diagonal entries in $C$; compare \cite{hansen2020}.
\end{example}

\begin{example}\label{ex:failureSec2}
Consider again the 2-cycle with an additional third node from the previous example.  Again, consider an arbitrary matrix  $C=(c_{ij})\in\PD_3$.  The kernel basis of \Cref{ex:kernel} restricted to the set of non-edges $E^c=\lbrace 1 \to 3,  3 \to 1, 2 \to 3, 3 \to 2 \rbrace$, namely
\begin{align*}
   H(\Sigma)_{E^c,\cdot}=
    \begin{pmatrix}
         -\Sigma_{23} & -\Sigma_{33} & 0 \\
        0 & \Sigma_{11} & \Sigma_{12} \\
        \Sigma_{13}  & 0 & -\Sigma_{33} \\
          0  & \Sigma_{12} & \Sigma_{22} 
  \end{pmatrix},
\end{align*}
is rank deficient for any $\Sigma\in\PD_3$ if and only if $\Sigma_{13}=\Sigma_{23}=0$. 
Adding this constraint to the Lyapunov equation yields $c_{13}=c_{23}=0$. Therefore, it follows from \Cref{theo:kernelcrit} that $\mathcal{M}_{G,C}$ is globally identifiable for any $C=(c_{ij})\in\PD_3$ in which $c_{13}$ and $c_{23}$ do not vanish simultaneously.

Observe that this provides a counterexample to \Cref{prop:non-simple} and \Cref{prop:subgraph:Cdiag} when dropping the diagonality assumption. To begin with, such $\mathcal{M}_{G,C}$ is an instance of a globally identifiable model associated to a non-simple graph. Moreover, the subgraph $H$ obtained by removing node 3 from $G$ defines a non-identifiable model for all positive definite volatility matrices by \Cref{lem:dim}. 

For the sake of completeness, note that, by \Cref{ex:failureTrek}, $c_{13}=c_{23}=0$ completely describes when the rank of $H(\Sigma)_{E^c,\cdot}$ drops for all $\Sigma\in\mathcal{M}_{G,C}$.
In other words, the model is non-identifiable if and only if $c_{13}=c_{23}=0$ and globally identifiable otherwise. 

\newpage

\end{example}

\section{Spectral description, kernel and factorization}\label{sec:app}
Here, we collect spectral properties of $A(\Sigma)$, derived more conveniently for its square $p\times p$ version
\[
\tilde{A}(\Sigma)=\Sigma \otimes I_p + (I_p \otimes \Sigma)K_p,
\]
which features in \Cref{lem:ASigma}.  We will then use this information to clarify that $\det(\Sigma)$ is a factor of $\det(A(\Sigma)_{\cdot,E})$ for complete graphs which have edge sets of size $|E|=p(p+1)/2$; see \Cref{cor:detSfactors}.

\begin{theorem}
\label{theo:Asigmaeigenvalues}
Let $\Sigma\in\PD_p$, and let $(\lambda_{i})_{i \in [p]}$ be its eigenvalues with corresponding orthogonal eigenvectors $(z_{i})_{i \in [p]}$. 
\begin{itemize}
    \item[(i)] The matrix $\tilde{A}(\Sigma)$ has rank $p(p+1)/2$, and \eqref{eq:HSigma} gives a basis for its kernel.
    \item[(ii)] The transposed matrix $\tilde{A}(\Sigma)^\top$ has rank $p(p+1)/2$, and a basis for its kernel is given by $\mathrm{vec}(e_i\otimes e_j-e_j\otimes e_i)$ for $1\le i< j\le p$.
    \item[(iii)] Counting with multiplicities, the $p(p+1)/2$ non-zero eigenvalues of $\tilde{A}(\Sigma)$ and of $\tilde{A}(\Sigma)^\top$ are given by the sums $\lambda_i+\lambda_j$ for $1\le i\le j\le p$ and for either matrix the associated set of orthogonal eigenvectors is $\mathrm{vec}(z_{i}\otimes z_{j}+z_{j}\otimes z_{i})$ for $1\le i\le j\le p$.
\end{itemize}

\begin{proof}
  (i) follows from~\eqref{eq:HSigma}, and (ii) follows from the symmetry of the Lyapunov (matrix) equation.
  
  For (iii), the claim about $\tilde A(\Sigma)$ follows from the calculation
  \begin{align*}
      &(z_{i}\otimes z_{j}+z_{j}\otimes z_{i})\Sigma+
      \Sigma (z_{i}\otimes z_{j}+z_{j}\otimes z_{i})^\top \\
      &= 
      \left[\lambda_j ( z_{i}\otimes z_{j})+\lambda_i (z_{j}\otimes z_{i}) \right]+ \left[\lambda_i ( z_{i}\otimes z_{j})+\lambda_j(z_{j}\otimes z_{i})\right]\\
      &=(\lambda_i+\lambda_j)(z_{i}\otimes z_{j}+z_{j}\otimes z_{i}).
  \end{align*}
  The transpose $\tilde A(\Sigma)^\top=\Sigma\otimes I_p+K_p(I_p\otimes\Sigma)$ encodes the Lyapunov equation with $M$ replaced by $M^\top$ and the claim about $\tilde A(\Sigma)^\top$ follows from the symmetry of the matrices $z_{i}\otimes z_{j}+z_{j}\otimes z_{i}$.  The orthogonality of the eigenvectors holds because 
  \[
  \mathrm{tr}\left((z_{i}\otimes z_{j}+z_{j}\otimes z_{i})(z_{k}\otimes z_{l}+z_{l}\otimes z_{k})\right)=0
  \]
  unless $\{i,j\}=\{k,l\}$.
\end{proof}   
\end{theorem}

As a consequence of \Cref{theo:Asigmaeigenvalues}, we can conclude information regarding the factorization of the determinant of $A(\Sigma)_{\cdot,E}$ when $|E|=p(p+1)/2$ such that $A(\Sigma)_{\cdot,E}$ is a square matrix.

\begin{corollary}
\label{cor:detSfactors}
Let $G=(V,E)$ be a directed graph with $V=[p]$ and $|E|=p(p+1)/2$. The polynomials $\det(\Sigma)$ and $\det(H(\Sigma)_{E^{c},\cdot})$ are factors of $\det(A(\Sigma)_{\cdot,E})$.
\begin{proof}
The zero set of the determinant $\det(\Sigma)$ is the set of singular symmetric matrices.  Since $\det(\Sigma)$ is an irreducible polynomial, every polynomial that vanishes at all singular matrices must be a polynomial multiple of $\det(\Sigma)$. Hence, it suffices to show that $\det(A(\Sigma)_{\cdot,E})=0$ for all singular matrices $\Sigma$.  So let $\Sigma$ be a singular matrix.  Then there exists an eigenvalue $\lambda_{i}=0$ with $i \in [p]$.  Using \Cref{theo:Asigmaeigenvalues} this implies that the eigenvalue $\lambda_{i}+\lambda_{i}$ of $\tilde{A}(\Sigma)$ is zero (the theorem is written for $\Sigma$ positive definite but the fact we used also holds for $\Sigma$ singular). Hence,  $\mathrm{rank}(\tilde{A}(\Sigma))\leq p(p+1)/2-1$ which implies that $\mathrm{rank}(A(\Sigma)_{\cdot,E})\leq p(p+1)/2-1$ and thus $\det(A(\Sigma)_{\cdot,E})=0$.

The fact that $\det(H(\Sigma)_{E^{c},\cdot})$ is a factor of $\det(A(\Sigma)_{\cdot,E})$ follows from the proof of \Cref{theo:kernelcrit}. 
\end{proof}
\end{corollary}

\end{document}